\documentclass[leqno]{article}

\usepackage{amsmath,amssymb,amsthm,amsfonts,bbm}
\usepackage{color}
\usepackage{epsf}
\usepackage{psfrag}
\usepackage{graphicx}
\usepackage{latexsym}

\theoremstyle{plain}
\newtheorem{Theorem}{Theorem}[section]
\newtheorem{Lemma}[Theorem]{Lemma}
\newtheorem{Proposition}[Theorem]{Proposition}
\newtheorem{Corollary}[Theorem]{Corollary}
\theoremstyle{definition}
\newtheorem{Assumption}[Theorem]{Assumption}
\theoremstyle{remark}
\newtheorem{Remark}[Theorem]{Remark}
\newtheorem{Example}[Theorem]{Example}

\def\bbC{{\mathbb C}}
\def\bbD{{\mathbb X}^1}

\def\bbR{{\mathbb R}}
\def\bbX{{\mathbb X}^0}
\def\bbZ{{\mathbb Z}}
\def\calC{{\mathcal C}}
\def\calL{{\mathcal L}}
\def\calO{{\mathcal O}}

\def\rmapp{{\mathrm{app}}}

\def\rmd{{\mathrm d}}

\def\rmm{{\mathrm m}}
\def\rmmax{{\mathrm{max}}}

\def\CNDF{C}
\def\CNIF{\eps}

\def\cc{{\mathrm{c.c.}}}

\def\Id{\operatorname{Id}}
\def\image{\operatorname{range}}

\def\eps{\varepsilon}

\numberwithin{equation}{section}

\begin{document}

\title{From Newton's cradle to the discrete $p$-Schr\"odinger equation}   

\author{
Brigitte Bid\'egaray-Fesquet$^{\, a}$,
Eric Dumas$^{\, b}$,
Guillaume James$^{\, a,c}$
\\
~\\
$\,^{a}$
Laboratoire Jean Kuntzmann, Universit\'e de Grenoble and CNRS\\
51 rue des Math\'ematiques, Campus de Saint Martin d'H\`eres,\\
BP 53, 38041 Grenoble cedex 9, France.
~\\~\\
$\,^{b}$
Institut Fourier, Universit\'e Grenoble 1 and CNRS\\
100 rue des Math\'ematiques, Campus de Saint Martin d'H\`eres,\\
BP~74, 38402 Saint Martin d'H\`eres cedex, France.
~\\~\\
$\,^{c}$
INRIA, Bipop Team-Project\\
ZIRST Montbonnot, 655 Avenue de l'Europe,\\ 
38334 Saint Ismier, France.
~\\~\\
{\small Email : Brigitte.Bidegaray@imag.fr,~edumas@ujf-grenoble.fr,~Guillaume.James@imag.fr}
}

\maketitle
\begin{abstract}
We investigate the dynamics of a 
chain of oscillators coupled by fully-nonlinear interaction potentials. This class of 
models includes Newton's cradle with Hertzian contact interactions between neighbors.
By means of multiple-scale analysis, we give a rigorous asymptotic 
description of small amplitude solutions over large times. 
The envelope equation leading to approximate solutions is a 
discrete $p$-Schr\"odinger equation. 
Our results include the existence of long-lived breather solutions 
to the original model. 
For a large class of localized initial conditions, we also estimate
the maximal decay of small amplitude solutions over long times.
\end{abstract}

\section{Introduction}

Newton's cradle is a nonlinear 
mechanical system consisting
of a chain of identical beads suspended from a bar by inelastic strings (see figure \ref{boules}).
All beads behave like pendula in the absence of contact
with nearest neighbors, i.e. they perform
time-periodic oscillations in a local confining potential 
$\Phi(x)=\frac{1}{2}x^2 + O(x^4)$ due to gravity.
More generally, the local potential $\Phi$ may account for different types of stiff attachments \cite{JKC11}
or an elastic matrix surrounding the beads \cite{lavigne,ahetal}.
Mechanical constraints between touching beads 
can be described by the Hertzian interaction potential
$V(r)=\frac{k}{1+\alpha}(- r)_+^{\alpha +1}$, where
$(a)_+ = \rm{max}(a,0)$, $k$ depends on the ball radius and material
and $\alpha = 3/2$. The dynamical equations 
read in dimensionless form \cite{Hutzler-Delaney-Weaire-MacLeod04}
\begin{equation}
\label{eq:cradle1}
\ddot x_n + \Phi^\prime (x_n ) = V'(x_{n+1}-x_n) - V'(x_n-x_{n-1}), \quad n\in\bbZ,
\end{equation}
$x_n$ being the horizontal displacement of the $n$th bead from its equilibrium position at which the pendulum is vertical.

\begin{figure}[h]
\begin{center}
\includegraphics[scale=0.11]{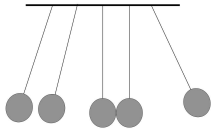}
\end{center}
\caption{\label{boules} 
Schematic representation of Newton's cradle.}
\end{figure}

Contact interactions between beads induce
a nonlinear coupling, which can lead to
complex dynamical phenomena like the propagation of solitary waves \cite{neste2,friesecke,mackay,english,stefanov,JKC11},
modulational instabilities \cite{James11,JKC11,boe} and the excitation of spatially localized
stationary (time-periodic) or moving breathers \cite{theo2,James11,sahvm,nolta12,JKC11}.
For small amplitude oscillations, it has been recently argued that such
dynamical phenomena can be captured by
the discrete $p$-Schr\"odinger (DpS) equation 
\begin{equation}
\label{dps}
2i \tau_0 \dot{A_n}=
(A_{n+1}-A_n)\, |A_{n+1}-A_n |^{\alpha -1} 
-(A_{n}-A_{n-1})\, |A_{n}-A_{n-1} |^{\alpha -1},
\end{equation}
with some time constant $\tau_0$ depending on $\alpha$.
More precisely, static breather solutions to (\ref{eq:cradle1}) were
numerically computed in \cite{JKC11} and compared to 
approximate solutions of the form 
\begin{equation}
\label{approx}
x_n^{\rm{app}} (t)=2\, \varepsilon\, \mbox{Re }[\, A_n(\varepsilon^{\alpha -1} t )\, e^{it} \, ] ,
\end{equation}
where $\varepsilon \ll 1$ and
$A_n$ denotes a breather solution to the DpS equation (\ref{dps}),
which depends on the slow time variable $\tau = \varepsilon^{\alpha -1} t$.
For small amplitudes, the Ansatz (\ref{approx}) was found to 
approximate breather solutions to (\ref{eq:cradle1}) with good 
accuracy \cite{JKC11}, 
and  the same property was established in \cite{James11} for periodic 
traveling waves. Moreover, a small amplitude velocity perturbation 
at the boundary of a semi-infinite chain (\ref{eq:cradle1}) generates a traveling breather 
whose profile is qualitatively close to (\ref{approx}), where 
$A_n$ corresponds to a traveling breather solution 
of the DpS equation \cite{sahvm,JKC11}.

In this paper, we put the relation between the original lattice (\ref{eq:cradle1})
and the DpS equation onto a rigorous footing. Our main result can be stated as follows
(a more precise statement extended to more general potentials 
is given in theorem \ref{th:classical}, section \ref{thms}).
Given a smooth ($C^2$) solution $A=(A_n(\tau ))_{n\in \mathbb{Z}}$ to (\ref{dps}) defined for $\tau \in [0,T]$, if
an initial condition for (\ref{eq:cradle1}) is
$O(\varepsilon^{\alpha})$-close to the Ansatz (\ref{approx}) at $t=0$,
then the corresponding solution to (\ref{eq:cradle1}) remains 
$O(\varepsilon^{\alpha})$-close
to the approximation (\ref{approx}) on $O(\varepsilon^{1-\alpha})$ time scales.
These error estimates hold in the usual sequence spaces
$\ell_p$ with $1\leq p \leq +\infty$.
In addition, if $A$ is a global and 
bounded solution to (\ref{dps}) in $\ell_p (\mathbb{Z})$ and $\delta \in (1,\alpha  )$ is fixed, 
then the same procedure yields $O(|\ln{\varepsilon}|\, \varepsilon^{\delta })$-close
approximate solutions
up to times $t=O(|\ln{\varepsilon}|\, \varepsilon^{1-\alpha})$
(theorem \ref{th:nonclassical}).
Moreover, similar estimates allow one
to approximate the evolution of {\em all} sufficiently small initial data to  (\ref{eq:cradle1})
in $(\ell_p (\mathbb{Z}))^2$ (theorems \ref{th:classicalbis} and \ref{th:nonclassicalbis}). 

Two applications of the above error estimates are presented.
Firstly, for all nontrivial solutions $A$ of the DpS equation in $\ell_2 (\mathbb{Z})$ we demonstrate that  
$\inf_{\tau \in\mathbb{R}}{\| A(\tau ) \|_{\infty}}>0$, i.e. 
solutions associated with localized (square-summable) initial
conditions do not completely disperse. Using this result 
and the above error bounds, we estimate
the maximal decay of small amplitude solutions to (\ref{eq:cradle1}) over long times
for a large class of localized initial conditions.
Secondly, from a breather existence theorem proved in \cite{js12} for the DpS equation,
we deduce the existence of stable small amplitude
``long-lived" breather solutions to equation (\ref{eq:cradle1}),
which remain close to time-periodic and spatially localized oscillations over long times.
This result completes a previous existence theorem for stationary breather solutions of (\ref{eq:cradle1})
proved in \cite{nolta12}, which was restricted to anharmonic on-site potentials $\Phi$
and small values of the coupling constant $k$.
More generally, the present justification of the DpS equation is also useful in the context of
numerical simulations of granular chains. Indeed, the DpS system is much easier to simulate than equation (\ref{eq:cradle1})
due to the fact that fast local oscillations have been averaged, 
which allows to perform
larger numerical integration steps.

Our results are in the same spirit as rigorous derivations
of the continuum cubic nonlinear Schr\"odinger or Davey-Stewartson
equations, which approximate the evolution of the envelope of 
slowly modulated normal modes in a large class of nonlinear lattices
\cite{gian,gian3,bambusi2,schneider2} and
hyperbolic systems \cite{DJMR95,JMR98,sch98,col02,CL04}.
In addition, our extension of the error bounds
up to times $\tau$ growing logarithmically in $\varepsilon$
(theorem \ref{th:nonclassical}) is reminiscent of refined
approximations of nonlinear geometric optics derived in \cite{lr00}.
A specificity of our result is the spatially discrete
character of the amplitude equation (\ref{dps}), which 
allows one to describe nonlinear waves with
rather general spatial behaviors 
(see theorem \ref{th:classicalbis} and \ref{th:nonclassicalbis}).
Another particular feature of our study is the fact that
potentials and nonlinearities can have limited smoothness, so that
high-order corrections seem hardly available.
 
The outline of the paper is as follows. 
In section \ref{evol}, we introduce a generalized version of system (\ref{eq:cradle1})
involving more general potentials, which is reformulated as a first order differential
equation in $(\ell_p (\mathbb{Z}))^2$. Section \ref{sec:linear} presents elementary
properties of periodic solutions to the linearized evolution problem, which will be used
in the subsequent analysis. 
The well-posedness of the Cauchy problem for the nonlinear evolution equation is
established in section \ref{wp}.
In section \ref{multiple}, we perform a formal multiple-scale analysis, yielding
approximate solutions to (\ref{eq:cradle1}) consisting of slow modulations of
periodic solutions. In this approximation, the leading term
(\ref{approx}) is supplemented by a higher-order corrector.
Some qualitative properties of the amplitude equation are detailed in section \ref{qual},
including well-posedness and the study of spatially localized solutions.
The main results on the justification of the multiple-scale analysis
from error bounds are stated in section \ref{thms}. The error bounds 
are derived in sections \ref{thms} and \ref{just}
(the later contains the proof of theorems \ref{th:classical}
and \ref{th:nonclassical}, which are mainly based on Gronwall estimates).
Section \ref{conclu}
provides a discussion of the above results and points out some open problems,
and some technical results are detailed in the appendix.

\section{\label{dynfor}Dynamical equations and multiple-scale analysis}

\subsection{\label{evol}Nonlinear lattice model}

We first introduce a more general version of system (\ref{eq:cradle1})
incorporating a larger class of potentials. 
We consider an interaction potential
$V\in\calC^2(\bbR,\bbR)$ of the form $V=V_\alpha+W$, where 
\begin{equation}
\label{as:Hertz}
V_\alpha (r) = 
\left\{
\begin{array}{c}
\frac{k_-}{1+\alpha}\, |r|^{\alpha +1} \mbox{ for } r\leq 0,\\
\frac{k_+}{1+\alpha}\, r^{\alpha +1} \mbox{ for } r\geq 0,
\end{array}
\right. 
\end{equation}
and $\alpha >1$, $(k_- , k_+ )\neq (0,0)$. 
In addition, $W$ is a higher order correction satisfying
\begin{equation}
\label{as:w}
W^\prime (0)  = 0,
\quad  W''(r) = \calO(|r|^{\alpha-1+\beta}) 
\quad \textrm{as }r\to0,
\end{equation}
for some constant $\beta>0$. 
We can therefore write $W^\prime (r)=|r|^\alpha\rho(r)$ where $\rho(r)=\calO(|r|^\beta)$ as $r\to0$.
Under the above assumptions, the principal part of $V$ satisfies
$V^\prime_\alpha(\lambda \, r) = \lambda^\alpha V_\alpha^\prime (r)$
for all $r\in\bbR$ and $\lambda>0$.
Note that one recovers 
the classical Hertzian potential by fixing $k_- = k >0$, $k_+=0$ and $W=0$.
The case $k_- = k_+$ and $W=0$ corresponds to an homogeneous even interaction potential.

In addition, the local potential $\Phi$ is assumed of the form
$\Phi(x)=\frac{1}{2}x^2 + \phi (x)$, where $\phi\in\calC^2(\bbR ,\bbR)$ satisfies  
\begin{equation}
\label{as:phi}
\phi^\prime (0)  = 0, \quad \phi''(x) = \calO(|x|^{\alpha-1+\gamma}) 
\quad \textrm{as }x\to0,
\end{equation}
for some constant $\gamma>0$.
We have therefore $\phi^\prime (x)=|x|^\alpha\chi(x)$ with $\chi(x)=\calO(|x|^\gamma)$ as $x\to0$.
The particular case of an harmonic on-site potential $\Phi$ is obtained by fixing $\phi =0$.

The dynamical equations read
\begin{equation}
\label{eq:cradle}
\ddot x_n + x_n = F(x)_n, \quad n\in\bbZ,
\end{equation}
where
\begin{equation*}
F(x)_n = V'(x_{n+1}-x_n) - V'(x_n-x_{n-1}) - \phi'(x_n), 
\quad n\in\bbZ.
\end{equation*}
From the above assumptions, the leading order nonlinear terms of (\ref{eq:cradle})
originate from the interaction potential $V_\alpha$.
We denote $x=(x_n)_{n\in\bbZ}$ and address solutions $x(t)\in\ell_p\equiv\ell_p(\bbZ,\bbR)$.
Throughout the paper we assume $p \in [1,\infty ]$ unless explicitly stated.

In what follows we reformulate
equation \eqref{eq:cradle} as a first order differential equation governing the $X=(x,\dot x)^T$ variable where $X(t)\in\ell_p^2$. 
This equation reads 
\begin{equation}
\label{eq:1storder}
\dot X = J\, X + G(X), 
\end{equation}
where
\begin{equation*}
J = \begin{pmatrix}0 & \Id \\ -\Id & 0\end{pmatrix}, 
\quad G(X_1, X_2) = \begin{pmatrix}0  \\ F(X_1) \end{pmatrix},
\end{equation*}
and $\Id$ is the identity map in $\ell_p$. Using the usual difference operators 
$(\delta^+x)_n=x_{n+1}-x_n$ and $(\delta^-x)_n=x_n-x_{n-1}$, we can write $F$ in a compact form:
\begin{equation*}
F(x)=\delta^+V'(\delta^-x)-\phi'(x).
\end{equation*} 
The Banach spaces $\ell_p^2$ are equipped
with the following norms
\begin{equation}
\label{normelp}
\|X\|_p  =   \left(\sum_{n\in\bbZ} (x_n^2+\dot{x}_n^2)^{p/2}  \right)^{1/p} \text{ if } 1\leq p<\infty, \quad
\|X\|_\infty  =  \sup_{n\in\bbZ}{(x_n^2+\dot{x}_n^2)^{1/2}}.
\end{equation}
The map $G$ is smooth and fully-nonlinear in $\ell_p^2$, as shown by the following lemma.
Below and in the rest of the paper, we use the abbreviations
{\em c.n.d.f.} for continuous non-decreasing functions and
{\em c.n.i.f.} for continuous non-increasing functions.

\begin{Lemma}
\label{estglp} 
The map $G\in\calC^1(\ell_p^2,\ell_p^2)$ satisfies
\begin{equation}
\label{estmapg}
\| G(X)\|_p = \calO (\| X \|_p^\alpha ), \quad
\| DG(X)\|_{\mathcal{L}(\ell_p^2)} = \calO(\| X \|_p^{\alpha-1} )
\end{equation} 
when $X\rightarrow 0$ in $\ell_p^2$.
Moreover, $G$ and $DG$ are bounded on bounded sets in $\ell_p^2$,
and there exist c.n.d.f. $C_G$, $C_D$
such that for all $X\in \ell_p^2$
\begin{equation}
\label{estgdg}
\| G(X)\|_p\leq C_G(\| X \|_\infty) \, \| X \|_p, \quad
\| DG(X)\|_{\mathcal{L}(\ell_p^2)}\leq C_D(\| X \|_\infty) .
\end{equation}
\end{Lemma}

\begin{proof}
It is a classical result that the functions $f=V'$ or $f=\phi'$ can be viewed as smooth
operators on $\ell_p(\bbZ,\bbR)$ via $(f(x))_n\stackrel{\text{\tiny def}}{=}f(x_n)$. 
The property $V'\in\calC^1(\ell_p,\ell_p)$ follows from
the continuous embedding $\ell_p \subset \ell_\infty$,
the fact that $V'\in\calC^1(\mathbb{R})$, $V^\prime (0)=0$
and $V''$ is uniformly continuous on compact intervals.
The property $\phi'\in\calC^1(\ell_p,\ell_p)$ follows from the same arguments.
These properties imply that
$F\in\calC^1(\ell_p,\ell_p)$ and $G\in\calC^1(\ell_p^2,\ell_p^2)$.

Moreover, standard estimates yield 
the following inequalities for $f = V^\prime , \phi^\prime$ and all $x \in \ell_p$,
\begin{equation}
\label{estvpphip}
\| Df(x) \|_{\mathcal{L}(\ell_p)}\leq M_f(\| x \|_\infty), \quad
\| f(x) \|_{p}\leq M_f(\| x \|_\infty) \, \| x \|_p ,
\end{equation}
where
$M_f(d)=\sup_{[0,d]}{|f^\prime|}$ 
defines a c.n.d.f. of $d \in \mathbb{R}^+$.
These estimates yield the bounds (\ref{estgdg}), hence
$G$ and $DG$ are bounded on bounded sets in $\ell_p^2$
due to the continuous embedding $\ell_p \subset \ell_\infty$.
Moreover, (\ref{estmapg}) follows by
combining (\ref{estvpphip}) with the bounds
$$
M_{V^\prime}(r)=\calO(r^{\alpha-1} ), \quad
M_{\phi^\prime}(r)=\calO(r^{\alpha-1+\gamma} ), \quad
r\rightarrow 0^+ ,
$$
which originate from
properties (\ref{as:Hertz}), (\ref{as:w}) and (\ref{as:phi}). 
\end{proof}

In section \ref{wp}, we prove the local well-posedness in $\ell_p^2$  
of the Cauchy problem associated with (\ref{eq:1storder}),
and its global well-posedness for $1\leq p\leq 2$ 
when $V \geq 0$ and $\phi \geq 0$. 
To obtain global solutions, we
use the fact that the Hamiltonian 
\begin{equation}
\label{hamic}
{H}=\sum_{n\in \mathbb{Z}}{\frac{1}{2}\dot{x}_n^2+\frac{1}{2}x_n^2+\phi (x_n )+V(x_{n+1}-x_n)}
\end{equation}
is a conserved quantity of (\ref{eq:cradle}).

\subsection{\label{sec:linear}Periodic solutions to the linearized equation}

In this section, we consider the time-periodic solutions to the linearized dynamical equations
which constitute the basic pattern slowly modulated in section \ref{multiple}.
We present some elementary properties of these solutions, solve
the nonhomogeneous linearized equations and compute the associated
solvability conditions. This result will be used in section
\ref{multiple} to derive the amplitude equation (\ref{dps}) as a solvability condition
and obtain the expression of a higher-order corrector to approximation (\ref{approx}),
following a usual multiple-scale perturbation scheme (see e.g. \cite{sulem}, section 1.1.3).

\vspace{1ex}

Equation (\ref{eq:cradle}) linearized at $x_n=0$ reads 
\begin{equation}
\label{eq:cradlelin}
\ddot x_n + x_n = 0, \quad n\in\bbZ,
\end{equation}
or equivalently 
\begin{equation}
\label{eq:cradlelin2}
X^\prime = J\, X. 
\end{equation}
Its solutions are $2\pi$-periodic and take the form
$$
\begin{pmatrix} x_n(t) \\ \dot{x}_n(t) \end{pmatrix}
= a_n e^{it} e_1 + \cc,
$$
where c.c. denotes the complex conjugate,
$e_{\pm1} = (1/{\sqrt2})(1 , \pm i)^T$
and $a_n = (x_n(0)-i \dot{x}_n(0))/\sqrt2$.
Moreover, assuming $X(0) \in \ell_p^2$ corresponds to imposing
$(a_n )_{n \in \mathbb{Z}}\in\ell_p(\bbZ,\bbC)$. 
In a more compact form, the solution $X=(x,\dot x)^T$ to equation (\ref{eq:cradlelin2}) with
initial condition $X(0) = X^0 =(x^0, \dot{x}^0)^T$ reads
\begin{eqnarray}
X(t)=e^{Jt} X^0 & = & \begin{pmatrix}\cos t & \sin t \\-\sin t & \cos t\end{pmatrix} X^0 \\
\label{expo}
&= &(\pi_1 X^0)\, e^{it} e_1 + (\pi_{-1} X^0)\, e^{-it} e_{-1}
\end{eqnarray}
where $\pi_{\pm 1}\in \mathcal{L}(\ell_p^2 , \ell_p(\bbZ,\bbC) )$ are defined by
$$
\pi_{\pm 1} \begin{pmatrix} x \\ \dot{x} \end{pmatrix} = \frac1{\sqrt2} (x\mp i \dot{x}),
$$
so that $\pi_1 X^0=(a_n )_{n \in \mathbb{Z}}$.

\begin{Remark}
\label{remcons}
One can notice that
$x_n^2+\dot{x}_n^2$ is a conserved quantity of (\ref{eq:cradlelin})
for all $n\in \mathbb{Z}$. 
With the choice of norm (\ref{normelp}), $\|X (t)\|_p$ is conserved along evolution for solutions to (\ref{eq:cradlelin2})
and it follows that $\|e^{Jt}\|_{\mathcal{L}(\ell_p^2)}=1$ for all $t\in \mathbb{R}$.
Moreover we have $\|X (t)\|_p = \|X^0\|_p=\sqrt{2}\, \|\pi_{\pm 1} X^0\|_p$.
\end{Remark}

In what follows we consider the nonhomogeneous linearized equation
\begin{equation}
\label{eq:linear}
\dot X = J\, X + U ,
\end{equation}
where $U\, : \, \mathbb{R} \rightarrow \ell_p^2$ is $2\pi$-periodic, and
derive compatibility conditions on $U$ allowing for the existence of $2\pi$-periodic
solutions to (\ref{eq:linear}).
We denote $S^1 = \mathbb{R}/2\pi\mathbb{Z}$
the periodic interval $[0,2\pi]$ and consider the function
spaces $\bbX=\calC^0(S^1,\ell_p^2)$ and $\bbD = \calC^1(S^1,\ell_p^2)$
endowed with their usual uniform topology.

\begin{Lemma}
\label{lm:linear_per}
Let $U\in\bbX$. 
The differential equation \eqref{eq:linear} has a solution $X\in\bbD $
if and only if
\begin{equation}
\label{compatplus}
\pi_{1} \left(\int_0^{2\pi} e^{- it}U(t)\ \rmd t\right) = 0.
\end{equation}
\end{Lemma}

\begin{proof}
Given $U\in\calC^0(\mathbb{R},\ell_p^2)$, the differential equation \eqref{eq:linear}
with initial condition $X(0) = X^0\in\ell_p^2$ has 
a unique solution $X\in\calC^1(\mathbb{R},\ell_p^2)$ given by the Duhamel integral
\begin{equation}
\label{eq:Duhamel}
X(t) = e^{Jt} X^0 + \int_0^t e^{J(t-s)} U(s)\ \rmd s .
\end{equation} 
Now let us assume $U\in\bbX$. In this case,
$X$ is $2\pi$-periodic iff $X(2\pi) = X(0)$.
Since $e^{J2\pi}=\Id$, this condition is realized when
\begin{equation*}
\int_0^{2\pi} e^{-Js} U(s)\ \rmd s = 0,
\end{equation*}
which is equivalent to condition 
\begin{equation}
\label{compat}
\pi_{\pm1} \left(\int_0^{2\pi} e^{\mp it}U(t)\ \rmd t\right) = 0
\end{equation}
due to identity (\ref{expo}).
Since $U$ is real and $\overline{\pi_{1}z}=\pi_{-1}\bar{z}$, condition 
(\ref{compat}) reduces simply to (\ref{compatplus}).
\end{proof}

The above results yield the following splitting of $\bbX$.

\begin{Lemma}
\label{split}
The operator $\partial_t-J$ maps $\bbD$ to $\bbX$, 
and we have the splitting
$\bbX= \ker(\partial_t-J)\, \oplus\,  \image(\partial_t-J)$.
The corresponding projector $P$ on $\ker(\partial_t-J)$ 
along $\image(\partial_t-J)$ reads
\begin{equation*}
PX = \zeta(X) e^{it} e_1 + \bar\zeta(X) e^{-it} e_{-1}, 
\end{equation*}
where $\zeta \in \mathcal{L}(\bbX ,\ell_p(\bbZ,\bbC))$ is defined by
\begin{equation*}
\zeta(X) = \frac1{2\pi} \int_0^{2\pi} e^{-it} \pi_1 X(t)\ \rmd t.
\end{equation*}
\end{Lemma}

\begin{proof}
It is clear that $P\in \mathcal{L}(\bbX)$ defines a projection
and $\image P=\ker(\partial_t-J)$ by identity (\ref{expo}).
Moreover, condition (\ref{compatplus}) shows that $\ker P= \image(\partial_t-J)$,
hence $\bbX= \image P \oplus \ker P= \ker(\partial_t-J)\, \oplus\,  \image(\partial_t-J)$. 
\end{proof}

Now one can deduce the following result from expression (\ref{eq:Duhamel}) and lemma \ref{split}.

\begin{Lemma}
\label{solpart}
For all $U\in\bbX$ satisfying (\ref{compat}) (or equivalently $P\, U=0$), 
equation \eqref{eq:linear} has a unique solution in
$\bbD \, \cap\, \image(\partial_t-J)$ given by
\begin{equation}
\label{solu}
X(t) = (K\, U)(t)= (I-P)\, \int_0^t e^{J(t-s)} U(s)\ \rmd s .
\end{equation}
Moreover, the linear operator $\mathcal{K}\stackrel{\text{\tiny def}}{=}K\, (I-P)\, : \bbX \rightarrow \bbD$ is bounded.
\end{Lemma}

\subsection{\label{wp}Well-posedness of the nonlinear evolution problem}

The following result ensures the local well-posedness 
of the Cauchy problem for (\ref{eq:1storder}) in $\ell_p^2$,
and its global well-posedness for positive potentials when $1\leq p\leq 2$.
In addition we derive a crude lower bound on the maximal existence times
for small initial data (estimate (\ref{tdef})).

\begin{Lemma}
\label{cauchypb}
For all initial condition $X(0)=X^0 \in \ell_p^2$,
equation (\ref{eq:1storder}) admits
a unique solution $X\in \calC^2 ((t_{-},t_{+}),\ell_p^2)$,
defined on a maximal interval of existence 
$(t_{-},t_{+})$ depending {\it a priori} on $X^0$
(with $t_{-}<0<t_{+}$).
In addition, there exists $T_0>0$ such that for $\| X(0) \|_p$ small enough
\begin{equation}
\label{tdef}
t_{+} > T_0\, \| X(0) \|_p^{1-\alpha }, \quad
t_{-} <- T_0\, \| X(0) \|_p^{1-\alpha } .
\end{equation}
Moreover, for $p\in [1,2]$,
if $V \geq 0$ and $\phi \geq 0$ then
$(t_{-},t_{+})=\mathbb{R}$
and $X\in L^\infty (\mathbb{R} , \ell_2^2)$.
\end{Lemma}

\begin{proof}
Since $G$ is $\calC^1$ in $\ell_p^2$,
it follows that the Cauchy problem for (\ref{eq:1storder}) is
locally well-posed in $\ell_p^2$ (see e.g. \cite{zeidler}, section 4.9).
More precisely, for all initial condition $X(0)=X^0 \in \ell_p^2$,
equation (\ref{eq:1storder}) admits
a unique solution $X\in \calC^1 ((t_{-},t_{+}),\ell_p^2)$,
defined on a maximal interval of existence 
$(t_{-},t_{+})$ depending {\it a priori} on $X^0$
(with $t_{-}<0<t_{+}$). Then a bootstrap argument
yields $X\in \calC^2 ((t_{-},t_{+}),\ell_p^2)$.

Now let us prove (\ref{tdef}) by Gronwall-type estimates. 
Equation~\eqref{eq:1storder} can be expressed in Duhamel form
\begin{equation}
\label{duha}
X(t) = e^{Jt} X(0) + \int_0^t e^{J(t-s)}G(X(s))\ \rmd s,
\end{equation}
where we recall that $\|e^{Jt}\|_{\mathcal{L}(\ell_p^2)}=1$.
In addition, by lemma \ref{estglp} there exist $R,\lambda >0$
such that 
\begin{equation}
\label{eg}
\| G(X)\|_p  \leq \lambda \| X \|_p^\alpha 
\end{equation}
for all $X\in \ell_p^2$ with $\| X \|_p  \leq R$. 
Denote by $\eps$ a small parameter and fix $\| X(0) \|_{p}=\eps$.
From the above properties we deduce
$$
\| X(t) \|_{p} \leq 
\eps + \lambda\, \int_0^t{\| X(s) \|_{p}^\alpha \, \rmd s}, \quad
0 \leq t \leq t_1, 
$$
where $t_1 = \sup \{ \, t \geq 0, \, {\| X(t) \|_{p}} \leq R \, \}$. This yields the estimate
\begin{equation}
\label{threestars}
\| X(t) \|_{p} \leq \varrho_{\eps , \lambda } (t), \quad
0 \leq t \leq t_1, 
\end{equation}
where $\varrho_{\eps , \lambda }$ is the solution to the differential equation
\begin{equation}
\label{defrho1}
\varrho^\prime = \lambda\, \varrho^\alpha , \quad \varrho (0)=\eps
\end{equation}
whose explicit form is
\begin{equation}
\label{defrho2}
\varrho_{\eps , \lambda } (t)=\eps\, [\, 1-(\alpha -1)\, \lambda\, \eps^{\alpha -1} t\, ]^{\frac{1}{1-\alpha}}.
\end{equation}
We note that $\varrho_{\eps , \lambda }$ blows up at $t=\eps^{1-\alpha }\, [\lambda\, (\alpha -1)]^{-1}$.
Consequently, we fix $\theta \in (0,1)$ and introduce 
$t_\eps = T_0\, \eps^{1-\alpha }$ with $T_0 = (1-\theta)\, [ \lambda\, (\alpha -1)]^{-1}$, and hence
for all $t\in [0, t_\eps ]$ we have
\begin{equation}
\label{plus}
\varrho_{\eps , \lambda } (t) \leq \varrho_{\eps , \lambda } (t_\eps )=\eps\, \theta^{\frac{1}{1-\alpha}}.
\end{equation}
Now let us assume
$\| X(0) \|_{p}=\eps \leq \theta^{\frac{1}{\alpha -1}}R$.
By estimates (\ref{threestars}) and (\ref{plus}) we have then
\begin{equation}
\label{estimx}
\| X(t) \|_{p} \leq \theta^{\frac{1}{1-\alpha }} \eps 
\quad  \mbox{ for }| t |  \leq t_\eps ,
\end{equation}
where the estimate for $t\leq 0$ is the same as for $t\geq 0$
due to the time-reversibility of (\ref{eq:1storder}) inherited from (\ref{eq:cradle}).
Consequently, $X(t)$ is defined and $\calO(\eps)$ in $\ell_p^2$ 
at least for $t\in [-t_\eps ,t_\eps ]$, which proves (\ref{tdef}).

In addition, the existence of a global solution to (\ref{eq:1storder}) in $\ell_2^2$
can be proved when $V \geq 0$ and $\phi \geq 0$, using the fact that
the Hamiltonian (\ref{hamic})
is a conserved quantity of (\ref{eq:cradle}). Indeed, for all initial condition $X(0)=X^0 \in \ell_2^2$
and for all $t\in (t_{-},t_{+})$ we have in that case
\begin{equation}
\label{bound}
\| X(t) \|_{2}^{2} \leq 2 {H} = \| X^0 \|_{2}^{2}+
2\sum_{n\in \mathbb{Z}}\big(\phi (x_n (0))+V(x_{n+1}(0)-x_n(0))\big),
\end{equation}
and thus $X\in L^\infty (\mathbb{R} , \ell_2^2)$.
Lemma \ref{estglp} ensures 
that $G$ and $DG$ are bounded on bounded sets in $\ell_2^2$.
Together with the uniform bound (\ref{bound}) on the solution $X$, this property implies that 
$(t_{-},t_{+})=\mathbb{R}$ (see e.g. \cite{rs}, theorem X.74).

Similar arguments can be used for global well-posedness in $\ell_p^2$ with
$p\in [1,2)$, except one uses the fact that $\| X(t) \|_{p}$ is bounded
on bounded time intervals. Indeed, the following estimate follows from equation
(\ref{duha}) and lemma \ref{estglp}
$$
\| X(t) \|_{p} \leq \| X(0) \|_{p} 
+ \int_0^t{C_G(\| X(s) \|_{\infty})\, \| X(s) \|_{p} \, \rmd s}, 
$$
where $C_G$ is a c.n.d.f.
Then using the continuous embedding $\ell_2 \subset \ell_\infty$ and estimate (\ref{bound}),
we find
$$
\| X(t) \|_{p} \leq \| X(0) \|_{p} + C_G(\sqrt{2 {H}})\,\int_0^t{ \| X(s) \|_{p} \, \rmd s}, 
$$
hence by Gronwall's lemma
$$
\| X(t) \|_{p} \leq e^{C_G(\sqrt{2 {H}})\, t}\, \| X(0) \|_{p}. 
$$
This shows that $\| X(t) \|_{p}$ is bounded
on bounded time intervals, which completes the proof.
\end{proof}

\subsection{\label{multiple}Multiple-scale expansion}

In this section, we perform a multiple-scale analysis in order to obtain
approximate solutions to equation (\ref{eq:1storder}). These approximations
consist of slow time modulations of small $2\pi$-periodic solutions to
the previously analyzed
linearized equation (\ref{eq:cradlelin2}).

To determine the relevant time scales, we denote by $\eps$ a small parameter and fix $\| X(0) \|_{p}=\eps$.
As previously seen in the proof of lemma \ref{cauchypb}, the solution
$X(t)$ of (\ref{eq:1storder})
is defined and $\calO(\eps)$ in $\ell_p^2$ at least on long time scales $t\in [-t_\eps ,t_\eps ]$
with $t_\eps=T_0\, \eps^{1-\alpha }$.
Considering the Duhamel form (\ref{duha}) of (\ref{eq:1storder})
when $t \approx t_\eps$,
the integral term at the right side of (\ref{duha})
is $\calO(t_\eps \, \eps^\alpha)$, so that
both terms are $\calO(\eps)$ and contribute ``equally" to $X(t)$.

We therefore consider the slow time $\tau=\eps^{\alpha-1}t$ in addition to the fast time variable $t$,
and look for slowly modulated periodic solutions
involving these two time scales:
\begin{equation*}
X(t) = \eps Y(\tau ,t)_{| \tau=\eps^{\alpha-1}t},
\end{equation*}
$Y$ being $2\pi$-periodic in the fast variable $t$.
Injecting this Ansatz in Equation~\eqref{eq:1storder}, we obtain 
\begin{equation}
\label{eq:2scales}
(\partial_t - J) Y - \frac1\eps G(\eps Y) + \eps^{\alpha-1} \partial_\tau Y = 0.
\end{equation}
Considering the class of potentials described in section \ref{evol}, one can write
\begin{equation}
\label{defreta}
\frac1\eps W'(\eps x) = \eps^{\alpha-1} {r_\eps(x)}, \quad 
\frac1\eps \phi'(\eps x) = \eps^{\alpha-1} {\eta_\eps(x)},
\end{equation}
with $r_\eps(x)=|x|^\alpha \rho(\eps x)= \calO(\varepsilon^{\beta}\, |x|^{\alpha + \beta})$ 
and $\eta_\eps(x) = |x|^\alpha \chi(\eps x)= \calO(\varepsilon^{\gamma}\, |x|^{\alpha + \gamma})$
when $\varepsilon \rightarrow 0$.
Hence we can split the nonlinear terms of (\ref{eq:2scales}) in the following way,
\begin{equation}
\label{decompg}
\frac1\eps G(\eps Y)  \equiv  \eps^{\alpha-1} G_\alpha(Y) + \eps^{\alpha-1} R_\eps(Y),
\end{equation}
where for all $Y=(Y_1 , Y_2)^T$,
\begin{equation*}
G_\alpha(Y) = \begin{pmatrix}0\\
\delta^+V_\alpha'(\delta^- Y_1)
\end{pmatrix} , \quad
R_\eps(Y) = \begin{pmatrix}0\\
\delta^+r_\eps(\delta^- Y_1) - \eta_\eps(Y_1)
\end{pmatrix} .
\end{equation*}
Moreover, for all $Y \in \ell_p^2$ we have $\lim_{\eps \rightarrow 0}R_\eps(Y) =0$
thanks to the assumptions made on $r_\eps$ and $\eta_\eps$
(see definition (\ref{defreta}) and
properties (\ref{as:w}) and (\ref{as:phi})).
More precisely, using the second estimate of (\ref{estvpphip}), there exist $C>0$ and a
c.n.i.f. $\varepsilon_0$
such that for all $Y \in \ell_p^2$ and $\varepsilon < \varepsilon_0 (\| Y \|_\infty )$,
\begin{equation}
\label{estre}
\| R_\eps(Y) \|_{p}\leq C\, \varepsilon^{\rm{min}(\beta , \gamma )}\, \| Y \|_p^{\alpha}\,
( \| Y \|_p^{\beta }+\| Y \|_p^{\gamma }).
\end{equation}
Now let us consider $\mathcal{Y}(\tau)=Y(\tau,.)\in \bbD$,
rewrite equation (\ref{eq:2scales}) as
\begin{equation}
\label{eq:2scales2}
(\partial_t - J) \mathcal{Y}= \eps^{\alpha-1} 
[\, 
G_\alpha(\mathcal{Y}) -  \partial_\tau \mathcal{Y} +  R_\eps(\mathcal{Y})
\, ]
\end{equation}
and look for approximate solutions to (\ref{eq:2scales2}) of the form
\begin{equation}
\label{expansion}
\mathcal{Y} = \mathcal{Y}_0 + \eps^{\alpha-1} \mathcal{Y}_1 + o(\eps^{\alpha-1}),
\end{equation}
where $\mathcal{Y}_{j}(\tau )\in \bbD$, $j=0,1$.
Inserting expansion (\ref{expansion}) in equation (\ref{eq:2scales2}) yields at leading order in $\eps$
\begin{equation}
\label{ordre0}
(\partial_t - J) \mathcal{Y}_0=0,
\end{equation}
i.e. $\mathcal{Y}_0(\tau ) \in \ker(\partial_t-J)$. Consequently, the principal part of the approximate solution takes the form
\begin{equation}
\label{y0}
( \mathcal{Y}_0 (\tau))(t) = a(\tau)e^{it}e_1 + \bar a(\tau)e^{-it}e_{-1},
\end{equation}
where $a(\tau)\in\ell_p (\mathbb{Z},\mathbb{C})$.
Similarly, identification at order $\eps^{\alpha-1}$ yields
\begin{equation}
\label{o1}
(\partial_t - J) \mathcal{Y}_1=G_\alpha(\mathcal{Y}_0) -  \partial_\tau \mathcal{Y}_0.
\end{equation}
According to lemma \ref{solpart},
this nonhomogeneous equation can be solved under the compatibility condition
$P\, [\, G_\alpha(\mathcal{Y}_0) -  \partial_\tau \mathcal{Y}_0 \, ]=0$, 
i.e. $\mathcal{Y}_0$ must satisfy the amplitude equation
\begin{equation}
\label{ampli}
\partial_\tau \mathcal{Y}_0= P\, G_\alpha(\mathcal{Y}_0).
\end{equation}
Then equation (\ref{o1}) becomes
\begin{equation}
\label{o1bis}
(\partial_t - J) \mathcal{Y}_1=(I-P)\, G_\alpha(\mathcal{Y}_0),
\end{equation}
which determines $\mathcal{Y}_1 (\tau )$ as a function of $\mathcal{Y}_0 (\tau )$, up
to an element of $\ker(\partial_t-J)$.
At our level of approximation, we can arbitrarily fix 
$\mathcal{Y}_1 (\tau )\in \image(\partial_t-J)$,
which yields according to lemma \ref{solpart}
\begin{equation}
\label{o1sol}
\mathcal{Y}_1 (\tau)=\mathcal{K}\, G_\alpha(\mathcal{Y}_0(\tau)).
\end{equation}
As a conclusion, we have obtained 
an approximate solution to equation \eqref{eq:1storder},
\begin{equation}
\label{approxsolu}
X_\rmapp^\eps(t)=\eps 
[\mathcal{Y}_0(\eps^{\alpha-1}t)](t) + \eps^{\alpha} [\mathcal{Y}_1(\eps^{\alpha-1}t)](t),
\end{equation}
where $\mathcal{Y}_0$
denotes a solution to the amplitude equation (\ref{ampli}) taking the form (\ref{y0})
and the corrector $\mathcal{Y}_1$ is defined by (\ref{o1sol}).

\vspace{1ex}

In section \ref{just} we justify the above formal multiple-scale analysis
by obtaining a suitable error bound on the approximate solutions.
The overall strategy is based on the following ideas. 
In section \ref{resi}, we
check that the approximate solution (\ref{approxsolu}) solves
\eqref{eq:1storder} up to an error that remains
$\calO(\varepsilon^{\alpha + \eta})$ for $t\in [0,t_{m}]$, with
$t_m=T\, \varepsilon^{1-\alpha}$ ($T>0$ fixed) and
$\eta = \rm{min}(\alpha -1 , \beta , \gamma)>0$
(this corresponds to the orders of the terms of (\ref{eq:2scales2}) 
neglected in the above analysis). 
From this result and Gronwall's inequality, we get the error estimate
$$
\| X_\rmapp^\eps(t) - X(t) \|_{L^\infty ([0,t_{m}],\ell_p^2)} = 
\calO(t_m\, \varepsilon^{\alpha + \eta}),
$$
provided 
$\| X_\rmapp^\eps(0) - X(0) \|_{p} = 
\calO(t_m\, \varepsilon^{\alpha + \eta})$
(section \ref{errorclassic}). 
Consequently, if $t_m$ were bounded then both terms of
approximation (\ref{approxsolu}) would be relevant for $t\in [0,t_{m}]$. 
However, in our case $t_m$ diverges for $\varepsilon \rightarrow 0$, hence
we get a larger error 
$$
\| X_\rmapp^\eps(t) - X(t) \|_{L^\infty ([0,t_{m}],\ell_p^2)} = \calO(\varepsilon^{1+\eta}).
$$
As a result, only the lowest-order term of approximation (\ref{approxsolu}) is
relevant on $\calO(\varepsilon^{1-\alpha})$ times, which finally yields 
the approximate solution to \eqref{eq:1storder}
\begin{equation}
\label{approxsoluo1}
X_{a}^\eps(t)=\eps 
[\mathcal{Y}_0( \eps^{\alpha-1}t )](t)
=\eps a( \eps^{\alpha-1}t  )e^{it}e_1 + \eps \bar a( \eps^{\alpha-1}t )e^{-it}e_{-1}.
\end{equation}

Let us examine more closely the amplitude equation satisfied by
$a(\tau)\in\ell_p (\mathbb{Z},\mathbb{C})$. 
The nonlinear term of (\ref{ampli}) can be explicitly computed
following the lines of \cite{James11}; 
see the appendix for details. 
More precisely, we have 
\begin{equation}
\label{eq:ampl}
i\partial_\tau a = \delta^+ f(\delta^- a),
\end{equation}
where
$$
(f(a))_n=\omega_0 \, a_n \, |a_n|^{\alpha -1},
$$
$$
\omega_0 
= (k_- + k_+)\, 
2^{\frac{\alpha -3}{2}} \, 
\frac{
\alpha \, \Gamma{(\frac{\alpha}{2})}
}
{
\sqrt{\pi} (\alpha +1)  \Gamma{(\frac{\alpha +1}{2})}
}
$$
and $\Gamma{(x)}=\int_{0}^{+\infty}{e^{-t}\, t^{x-1}\, \rmd t}$ 
denotes Euler's Gamma function.
Equation (\ref{eq:ampl}) reads component-wise
\begin{equation}
\label{dpsa}
i  \frac{\partial a_n}{\partial \tau}=\omega_0\, (\Delta_{\alpha +1}a)_n, \ \
n\in \mathbb{Z},
\end{equation}
where the nonlinear difference operator
$$
(\Delta_{\alpha +1}a)_n=(a_{n+1}-a_n)\, |a_{n+1}-a_n |^{\alpha -1} -
(a_{n}-a_{n-1})\, |a_{n}-a_{n-1} |^{\alpha -1}
$$
is the discrete $(\alpha +1)$-Laplacian. 

\subsection{\label{qual}Qualitative properties of the amplitude equation}

In this section we establish the well-posedness of the differential equation (\ref{dpsa}),
point out some invariances and conserved quantities
yielding global existence results, and study the existence of 
spatially localized solutions which do not decay when $\tau \rightarrow +\infty$.

\subsubsection{\label{conserv}Conserved quantities and well-posedness}

Let $a\in \calC^1 ((T_{\rm{min}},T_{\rm{max}}),\ell_{\alpha +1}(\mathbb{Z},\mathbb{C}))$ denote a 
solution to (\ref{eq:ampl}) defined on some time interval $(T_{\rm{min}},T_{\rm{max}})$.
One can readily check that the quantity
$$
{\| \delta^{+} a \|}_{\alpha +1}^{\alpha +1} = \sum_{n\in \mathbb{Z}}{|a_{n+1}-a_n|^{\alpha +1}}
$$
is conserved along evolution. This property is linked with the Hamiltonian structure of equation (\ref{dpsa}),
which can be formally written
\begin{equation}
\label{hamform}
\frac{\partial a_n}{\partial \tau}=i\, \frac{\partial \mathcal{H} }{\partial \bar{a}_n}, \ \
n\in \mathbb{Z}, \quad \text{with} \quad 
\mathcal{H} =
\frac{2\omega_0}{\alpha +1}\, {\| \delta^{+} a \|}_{\alpha +1}^{\alpha +1}.
\end{equation}
More precisely, setting 
$$
a=\pi_{ 1} \begin{pmatrix} q \\ p \end{pmatrix} = \frac1{\sqrt2} (q- i p),
$$
the solutions $a\in \calC^1 ((T_{\rm{min}},T_{\rm{max}}),\ell_{2}(\mathbb{Z},\mathbb{C}))$ of (\ref{eq:ampl})
correspond to solutions $(q,p)\in \calC^1 ((T_{\rm{min}},T_{\rm{max}}),\ell_{2}^2)$ of the Hamiltonian
system 
$$
\begin{pmatrix} \dot{q} \\ \dot{p} \end{pmatrix} = J\, \nabla\mathcal{H}(q,p),
$$
where
$$
\mathcal{H}(q,p)=\frac{2\omega_0}{\alpha +1}\,  \sum_{n\in \mathbb{Z}}
{{\left[\, \left(\frac{q_{n+1}-q_n}{\sqrt{2}}\right)^2 + \left(\frac{p_{n+1}-p_n}{\sqrt{2}}\right)^2 \, \right]}^{\frac{\alpha +1}{2}}}
$$
is defined on the real Hilbert space $\ell_{2}^2$. 

Equation (\ref{dpsa}) admits the gauge invariance
$a_n \rightarrow a_n \, e^{i \varphi}$, the translational invariance $a_n \rightarrow a_n + c$ and
a scale invariance, since any
solution $(a_n)_{n\in\mathbb{Z}}$ of (\ref{dpsa}) generates a one-parameter family of solutions
$\left( \varepsilon\, a_n (|\varepsilon|^{\alpha -1}\, \tau ) 
\right)_{n\in\mathbb{Z}}$, $\varepsilon\in \mathbb{R}$. 
Several conserved quantities of (\ref{dpsa}) can be associated to these
invariances via Noether's theorem. The scale invariance and the invariance by time translation
correspond to the conservation of $\mathcal{H}$. 
The gauge invariance yields the conserved quantity
$$
\| a \|_2^2 =\sum_{n\in \mathbb{Z}}{|a_n|^2}
$$
whenever $a\in \calC^1 ((T_{\rm{min}},T_{\rm{max}}),\ell_2(\mathbb{Z},\mathbb{C}))$.
In the same way, the translational invariance yields the additional
conserved quantity
$$
\mathcal{P}=\sum_{n\in \mathbb{Z}}{a_{n}}
$$
provided $a\in \calC^1 ((T_{\rm{min}},T_{\rm{max}}),\ell_1(\mathbb{Z},\mathbb{C}))$.

\vspace{1ex}

In the sequel we use the notation $\ell_p=\ell_p (\mathbb{Z},\mathbb{C})$. 
The following lemma ensures the local well-posedness of equation
(\ref{dpsa}) in $\ell_p$ for $p\in [1,+\infty]$ and its global well-posedness 
for $p\in [1,1+\alpha]$ (in particular for $p=2$).

\begin{Lemma}
\label{lm:estim_a}
Let $a^0\in\ell_p$. 
Equation \eqref{dpsa} with initial data $a(0)=a^0$ 
admits a unique solution $a\in \calC^2 ((T_{\rm{min}},T_{\rm{max}}),\ell_p)$,
defined on a maximal interval of existence 
$(T_{\rm{min}},T_{\rm{max}})$ depending {\it a priori} on $a^0$
(with $T_{\rm{min}}<0<T_{\rm{max}}$). One has in addition
\begin{equation}
\label{tdefa}
T_{\rm{max}} \geq T_1\, \| a^0 \|_p^{1-\alpha }, \quad
T_{\rm{min}} \leq - T_1\, \| a^0 \|_p^{1-\alpha }, 
\end{equation}
with $T_1=[(\alpha -1)\, 2^{\alpha +1} \omega_0]^{-1}$.
Moreover, $(T_{\rm{min}},T_{\rm{max}})=\mathbb{R}$ if $p\in [1,\alpha +1]$.
\end{Lemma}

\begin{proof}
Since $\alpha >1$ we have $\Delta_{\alpha +1} \in \calC^1(\ell_p,\ell_p)$
and thus the Cauchy problem for (\ref{dpsa}) is
locally well-posed in $\ell_p$.
Therefore, for all initial condition $a^0 \in \ell_p$,
equation (\ref{dpsa}) admits
a unique maximal solution $a\in \calC^1 ((T_{\rm{min}},T_{\rm{max}}),\ell_p)$,
and a bootstrap argument yields then
$a\in\calC^2((T_{\rm{min}},T_{\rm{max}}),\ell_p)$.

To prove estimates (\ref{tdefa}), we rewrite (\ref{dpsa}) in the form
\begin{equation}
\label{integdps}
i a(\tau ) = i a^0+\delta^+ \int_0^\tau{f(\delta^- a(s))\, \rmd s}.
\end{equation}
Since $\| f(a) \|_{p} \leq \omega_0 \, \| a \|_{p}^\alpha$, we get
$$
\| a(\tau ) \|_{p} \leq  \| a^0 \|_{p}
+2^{\alpha+1}\, \omega_0\, \int_0^\tau{ \| a(s) \|_{p}^\alpha \, \rmd s},
\quad
0 \leq \tau < T_{\rm{max}}.
$$
As in the proof of lemma \ref{cauchypb}, 
a Gronwall-type estimate yields then 
\begin{equation}
\label{bornea}
\| a(\tau ) \|_{p} \leq \varrho_{\eps , \lambda } (\tau ), \quad
0 \leq \tau < \eps^{1-\alpha}\, T_1 ,
\end{equation}
for the parameter choice
$\eps = \| a^0 \|_{p}$ and
$\lambda = 2^{\alpha+1}\, \omega_0$,
where $\varrho_{\eps , \lambda }(\tau )$ is the solution to
the differential equation (\ref{defrho1}) with explicit form
(\ref{defrho2}) defined up to $\tau = \eps^{1-\alpha}\,T_1$. Consequently,
bound (\ref{bornea}) yields the first estimate of (\ref{tdefa}),
and the estimate for $\tau \leq 0$ is the same as for $\tau \geq 0$
owing to the invariance $a(\tau ) \rightarrow \bar{a}(-\tau )$ of (\ref{dpsa}).

Global well-posedness in $\ell_p$ for $p\in [1,\alpha +1]$ follows from 
the fact that $\Delta_{\alpha +1}$, $D\Delta_{\alpha +1}$
are bounded on bounded sets in $\ell_p$ and $\| a(\tau) \|_{p}$ is bounded
on bounded time intervals. To prove this second property, 
we deduce from (\ref{integdps})
\begin{eqnarray*}
\| a(\tau ) \|_{p} &\leq  &\| a^0 \|_{p}
+2\, \omega_0\, \int_0^\tau{ \| \delta^- a(s) \|_p\, \| \delta^- a(s) \|_{\infty}^{\alpha -1} \, \rmd s}, \\
&\leq  &\| a^0 \|_{p}
+2\, \omega_0\, \int_0^\tau{ \| \delta^- a(s) \|_{p}\, \| \delta^- a(s) \|_{\alpha +1}^{\alpha -1} \, \rmd s},\\
&\leq  &\| a^0 \|_{p}
+4\, \omega_0\, \| \delta^- a^0 \|_{\alpha +1}^{\alpha -1} \, \int_0^\tau{ \|  a(s) \|_{p}\,  \rmd s},
\end{eqnarray*}
where we have used the fact that $\| f(a) \|_{p} \leq \omega_0 \, {\| a \|}_{p} \, {\| a \|}_{\infty}^{\alpha -1}$
and $\| \delta^- a(\tau) \|_{\alpha +1}$ is conserved.
Now we have by Gronwall's lemma
\begin{equation}
\label{bounda}
\| a(\tau ) \|_{p} \leq e^{\sigma \, \tau}\, \| a^0 \|_{p}, \quad 
\sigma = 2^{\alpha +1}\, \omega_0\, {\| a^0 \|}_{\alpha +1}^{\alpha -1},
\end{equation}
and the proof is complete.
\end{proof}

\begin{Remark}
The bound (\ref{bounda}) can be improved for $p=2$ and $p=\alpha +1$, since
$\| a(\tau) \|_2=\| a^0 \|_2$ is bounded and a sharper estimate
can be deduced from (\ref{integdps})~:
$$
{\| a(\tau ) \|}_{\alpha +1} \leq {\| a^0 \|}_{\alpha +1}
+2\omega_0\, \tau\,  {\| \delta^- a^0 \|}_{\alpha +1}^\alpha .
$$
\end{Remark}

\subsubsection{\label{spatloc}Spatially localized solutions}

Another important feature of equation (\ref{dpsa})
is the absence of scattering for square-summable solutions.
More precisely, the following result ensures that all (nontrivial)
solutions to (\ref{dpsa}) in $\ell_2$ satisfy
$\inf_{\tau \in\mathbb{R}}{\| a(\tau ) \|_{\infty}}>0$, which implies that
they do not completely disperse.
The proof is based on the conservation of $\ell_2$ norm and energy,
an idea introduced in \cite{k08} in the context of the
disordered discrete nonlinear Schr\"odinger equation.

\begin{Lemma} 
\label{minLinftynorm} 
Let $a^0 \in \ell_2$ with $a^0 \neq 0$ 
and $a\in \calC^1 (\mathbb{R},\ell_2)$ denote
the solution to \eqref{dpsa} with $a(0)=a^0$. 
Then we have
$$
\forall \tau\in\mathbb{R}, \quad 
\|a(\tau)\|_\infty \geq 
\left( 
\frac{\|\frac12 \delta^+ a^0\|_{\alpha+1}^{\alpha+1}}{\|a^0\|_2^2}
\right)^\frac{1}{\alpha-1} .
$$
\end{Lemma}

\begin{proof}
Simply use the conserved quantities $\|\delta^+ a\|_{\alpha+1}$ 
and $\|a\|_2$ from section \ref{conserv}, and estimate thanks to 
the triangle and interpolation inequalities:  
\begin{eqnarray*}
\|\delta^+ a^0\|_{\alpha+1} & = & \|\delta^+ a(\tau)\|_{\alpha+1} \\ 
& \leq & 2 \|a(\tau)\|_{\alpha+1} \\
& \leq & 2 \|a(\tau)\|_\infty^{1-\frac{2}{\alpha+1}} 
\|a(\tau)\|_2^\frac{2}{\alpha+1} 
= 2 \|a(\tau)\|_\infty^{1-\frac{2}{\alpha+1}} 
\|a^0\|_2^\frac{2}{\alpha+1} . 
\end{eqnarray*}
\end{proof}

When $a^0$ is restricted to some subset $\mathcal{C}$ of $\ell_2$,
the following result provides a simpler estimate involving the constant
\begin{equation}
\label{defi}
\mathcal{I}(\mathcal{C})=\inf
\left\{
Q_\alpha (a), \, a\in \mathcal{C}, \, a \neq 0
\right\} ,
\end{equation}
where 
\begin{equation}
\label{defq}
Q_\alpha (a)= \frac{\| \delta^+ a\|_{\alpha+1}}{\|a\|_2} .
\end{equation}

\begin{Corollary}
\label{colnd}
Keep the notations of lemma \ref{minLinftynorm} and assume $a^0 \in \mathcal{C} \subset \ell_2$.
Then we have
\begin{equation}
\label{estsimp}
\forall \tau\in\mathbb{R}, \quad 
\|a(\tau)\|_\infty \geq 
{\left(
\frac{1}{2}\, \mathcal{I}(\mathcal{C})
\right)}^\frac{\alpha +1}{\alpha-1}
\, \|a^0\|_2 .
\end{equation}
\end{Corollary}

\begin{proof}
Use lemma \ref{minLinftynorm} and the identity 
$$
\left( 
\frac{\|\frac12 \delta^+ a^0\|_{\alpha+1}^{\alpha+1}}{\|a^0\|_2^2}
\right)^\frac{1}{\alpha-1} 
=
\|a^0\|_2 \, {\left(\frac{1}{2}\, Q_\alpha (a^0) \right)}^\frac{\alpha +1}{\alpha-1}.
$$
\end{proof}

\begin{Remark}
The above result is useless for $\mathcal{C}=\ell_2$ since $\mathcal{I}(\ell_2 )=0$.
Indeed, considering the sequence $a^N = \mathbbm{1}_{\{ 1,\ldots ,N \}}$
(with $\mathbbm{1}$ denoting the indicator function)
one can check that $Q_\alpha (a^N) = 2^\frac{1}{\alpha +1} N^{-1/2} \rightarrow 0$
as $N\rightarrow +\infty$.
\end{Remark}

\begin{Remark}
\label{remi}
If $\mathcal{C}$ is a finite-dimensional linear subspace of 
$\ell_2$, then $\mathcal{I}( \mathcal{C} )>0$
($Q_\alpha$ is the ratio of two equivalent norms on $\mathcal{C}$).
Moreover, if $\mathcal{I}( \mathcal{C} )>0$ on some subspace $\mathcal{C}$
of $\ell_2$ then the norms $\| \, \|_\infty$ and $\| \, \|_2$ are equivalent on $\mathcal{C}$
(this follows from the case $\tau=0$ of (\ref{estsimp}) and the continuous embedding
$\ell_2 \subset \ell_\infty$).
\end{Remark}

Lemma \ref{minLinftynorm} and corollary \ref{colnd} show that
all square-summable localized solutions do not decay as $\tau \rightarrow \pm \infty$.
One of the simplest type of 
localized solutions to \eqref{dpsa} corresponds to
time-periodic oscillations (discrete breathers), which have been studied
in a number of works (see \cite{js12} and references therein). 
Equation (\ref{dpsa}) admits time-periodic solutions of the form
\begin{equation}
\label{solper}
a_n (\tau ) = \varepsilon\,  v_n\, e^{i\, \omega_0 \, |\varepsilon|^{\alpha -1}\, \tau},
\end{equation}
where $v=(v_n)_{n \in \mathbb{Z}}$ is a real sequence and $\varepsilon\in \mathbb{R}$
an arbitrary constant, if and only if $v$ satisfies
\begin{equation}
\label{dpsstat}
{v_n}= -(\Delta_{\alpha +1}v)_n, \ \
n\in \mathbb{Z}.
\end{equation} 
In particular, nontrivial solutions to (\ref{dpsstat}) 
satisfying $\lim_{n\rightarrow \pm\infty}{v_n}=0$
correspond to breather solutions to (\ref{dpsa})
given by (\ref{solper}). The following existence theorem
for spatially symmetric breathers has been proved in \cite{js12}
using a reformulation of (\ref{dpsstat}) as a two-dimensional mapping.

\begin{Theorem}
\label{homoclinic}
The stationary DpS equation (\ref{dpsstat}) admits
solutions $v_n^i$ ($i=1,2$) satisfying 
$$
\lim_{n\rightarrow \pm\infty}v_n^i=0,
$$ 
$$
(-1)^n\, v_{n}^i >0, \quad  
|v_n^i | > |v_{n-1}^i | \quad \text{for all } \, n\leq 0, 
$$
$$
\text{and} \quad 
v_n^1 =v_{-n}^1,  \ \ \ 
v_n^2 = -v_{-n+1}^2, 
 \ \ \  \mbox{ for all } n \in \mathbb{Z}.
$$
Furthermore, for all $q\in(0,1)$, there exists $n_0\in\mathbb{N}$ 
such that the above-mentioned solutions $v_n^i$ satisfy, 
for $i=1,2$: 
$$
\forall n \geq n_0, \quad |v_n^i| \leq q^{1+\alpha^{n-n_0}} . 
$$
\end{Theorem}

\begin{Remark}
These solutions are thus doubly exponentially decaying, 
so that they belong to $\ell_p$ for all $p\in[1,\infty]$. 
\end{Remark}

One may wonder if results analogous to lemma \ref{minLinftynorm}
and theorem \ref{homoclinic} hold true for the original lattice (\ref{eq:cradle}).
The proofs of the above results heavily rely on the gauge invariance of \eqref{dpsa}
which implies the conservation of $\| a(\tau ) \|_2$. Such properties are not available
for system (\ref{eq:cradle}), hence the same methodology cannot be directly applied  
to Newton's cradle. These problems will be solved in the next section through the
justification of approximation (\ref{approxsoluo1}) on long time scales.

\subsection{\label{thms}Error bounds and applications}

In this section, we give several error bounds in order to justify the
expansions of section \ref{multiple}, for small amplitude solutions and
long (but finite) time intervals. From these error bounds, we also infer 
stability results for long-lived
breather solutions to the original lattice model, 
as well as lower bounds for the amplitudes of small solutions valid over long times
(see section \ref{applithm}). 

\subsubsection{\label{thmO(eps(1-alpha))}
Asymptotics for times $\mathcal{O}(1/\eps^{\alpha-1})$}

In theorem \ref{th:classical} below, one considers any solution 
$a \in \calC^2 ([0,T],\ell_p)$
to equation \eqref{eq:ampl} and constructs a family $X_a^\eps$ of approximate solutions to \eqref{eq:1storder},
whose amplitudes are $\mathcal{O}(\eps )$ and determined by $a$ and $\eps$. These approximate solutions
are $\mathcal{O}(\eps^{1+\eta})$-close to exact solutions for 
some constant $\eta >0$ specified below and
$t\in[0,T/\eps^{\alpha-1}]$. 
The proof of theorem \ref{th:classical} is detailed in section \ref{errorclassic}.

\begin{Theorem}
\label{th:classical}
Let $\eta=\min(\alpha-1,\beta,\gamma)$ and
fix two constants $C_{\rm{i}} ,T >0$.
There exist a c.n.i.f. $\eps_T>0$ and a c.n.d.f. $C_T \geq C_{\rm{i}}$ such that
the following holds:\\
For all solution $a\in \calC^2 ([0,T],\ell_p)$ to equation \eqref{eq:ampl}
with $N\stackrel{\text{\tiny def}}{=}\|a\|_{L^\infty([0,T],\ell_p)}$ and
for all $\eps\leq\eps_T(N)$, we define
\begin{equation}
\label{ansatzthm}
X_a^\eps(t)= \frac\eps{\sqrt2} a(\eps^{\alpha-1}t)\, e^{it}\begin{pmatrix}1\\i\end{pmatrix} + \cc
\end{equation}
Then, for all
$X^0\in\ell_p^2$ satisfying 
$$
\|X^0-X_a^\eps(0)\|_p \leq C_{\rm{i}}\eps^{1+\eta}, 
$$
the solution $X(t)\in \ell_p^2$ of
equation \eqref{eq:1storder} with $X(0)=X^0$ 
is defined at least for $t\in[0,T/\eps^{\alpha-1}]$ and satisfies
\begin{equation}
\label{erroran}
\|X(t)-X_a^\eps(t)\|_p \leq C_T(N)\, \eps^{1+\eta} 
\textrm{ for all }t\in[0,T/\eps^{\alpha-1}].
\end{equation}
\end{Theorem}

\vspace{1ex}

\begin{Remark}
The case $\phi =0$ of equation (\ref{eq:cradle}) (harmonic on-site potential $\Phi$)
corresponds to fixing $\gamma = + \infty$, which yields
$\eta=\min(\alpha-1,\beta)$. Similarly, the case
$W=0$ (pure Hertzian-type interaction potential $V=V_\alpha$)
is obtained with $\beta = + \infty$ and $\eta=\min(\alpha-1,\gamma)$.
\end{Remark}

\vspace{1ex}

\begin{Remark}
By lemma \ref{cauchypb}, the solution $X$ of
equation \eqref{eq:1storder} satisfies 
$X= (x, \dot{x})^T \in \calC^2 ([0,T/\eps^{\alpha-1}],\ell_p^2)$,
hence $x \in \calC^3 ([0,T/\eps^{\alpha-1}],\ell_p)$. Moreover,
it follows from lemma \ref{lm:estim_a} that
$X_a^\eps = (x_a^\eps, y_a^\eps )^T \in \calC^2 ([0,T/\eps^{\alpha-1}],\ell_p^2)$,
hence $x_a^\eps \in \calC^2 ([0,T/\eps^{\alpha-1}],\ell_p)$.
Consequently, the exact solution $x$ of \eqref{eq:cradle}
is generally more regular than its approximation $x_a^\eps$, 
which is an unusual property in the context of modulation equations.
\end{Remark}

\vspace{1ex}

\begin{Example}
In the case $\alpha = 3/2$ (as for the classical Hertz force), if $W$ and $\phi$ are
$\mathcal{C}^3$ at both sides of the origin, then $\beta , \gamma \geq 1/2$
in (\ref{as:w}) and (\ref{as:phi})
and we have by theorem \ref{th:classical}
$$
\|X(t)-X_a^\eps(t)\|_p \leq C_T(N)\, \eps^{3/2} 
\textrm{ for all }t\in[0,T\, \eps^{-1/2}].
$$
\end{Example}

\vspace{1ex}

As a corollary of theorem \ref{th:classical} and previous estimates 
on the solutions to the amplitude equation \eqref{eq:ampl},
one obtains theorem \ref{th:classicalbis} below.
Roughly speaking, 
for {\em all} sufficiently small initial data $X^0 \in \ell_p$, 
this theorem provides an approximation
$X_A$ of the solution to \eqref{eq:1storder} 
(with amplitude given by a solution to \eqref{eq:ampl})
valid on $\mathcal{O}(\| X^0 \|_p^{1-\alpha})$ time scales.

\begin{Theorem}
\label{th:classicalbis}
Let $\eta=\min(\alpha-1,\beta,\gamma)$ and $C_{\rm{i}}>0$.
There exists $T_{\rm m}\in (0, +\infty ]$ such that for all $T \in (0, T_{\rm m} )$,
there exist $\tilde\eps_T>0$ and $\tilde{C}_T \geq C_{\rm{i}}$ 
such that the following properties hold.
For all $X^\ast = (x^\ast, y^\ast)^T\in \ell_p^2$ such 
that $0<\eps \stackrel{\text{\tiny def}}{=}  \| X^\ast \|_p\leq \tilde\eps_T$,
for all $X^0 \in \ell_p^2$ satisfying $\|X^0-X^\ast\|_p \leq C_{\rm{i}}\eps^{1+\eta}$,
the solution $X(t)\in \ell_p^2$ to 
equation \eqref{eq:1storder} with $X(0)=X^0$ 
is defined at least for $t\in[0,T\eps^{1-\alpha}]$.
This solution satisfies
\begin{equation}
\label{erreur}
\|X(t)-X_A(t)\|_p \leq \tilde{C}_T\eps^{1+\eta} 
\textrm{ for all }t\in[0,T\eps^{1-\alpha}],
\end{equation}
where
\begin{equation*}
X_A(t)= \frac{1}{\sqrt2} A(t)\, e^{it}\begin{pmatrix}1\\i\end{pmatrix} + \cc ,
\end{equation*}
and $A\in \calC^2 ([0,T\eps^{1-\alpha}],\ell_p)$ is the solution to equation \eqref{eq:ampl}
with $\mathcal{O}(\eps )$ initial condition 
$A (0)= (x^\ast-i \, y^\ast)/\sqrt2$.
Moreover, if $p\in [1,1+\alpha ]$ then $T_{\rm m}= +\infty $.
\end{Theorem}

\begin{proof}
By lemma \ref{lm:estim_a}, there exists 
a unique maximal solution $A$ of \eqref{eq:ampl} with initial condition $A(0)$
defined as above from $X^\ast$. 
Due to the scale invariance of \eqref{eq:ampl}, the initial condition
$a(0)=\eps^{-1}\, A(0)$ yields the solution $a(\tau )=\eps^{-1}\, A(\eps^{1-\alpha}\, \tau)$.
Since $\| a(0) \|_p = \sqrt{2}/2$ by construction, lemma \ref{lm:estim_a}
ensures that $a(\tau )$ is defined and bounded in $\ell_p$ for
$\tau \in [0,T]$ whenever $T<T_1\, 2^{(\alpha -1)/2}$. 
In addition this property is true for all $T \in (0,+\infty )$ 
when $p\in [1,1+\alpha ]$. This leads us to define $T_{\rm m}=T_1\, 2^{(\alpha -1)/2}$
for $p>1+\alpha$ and $T_{\rm m}= +\infty $ for $p\in [1,1+\alpha ]$.

Now let us consider $T < T_{\rm m}$ being fixed,
so that $A(t )$ is defined and bounded in $\ell_p$ for
$t \in [0,T\eps^{1-\alpha}]$. 
Using either bounds (\ref{bornea}) or (\ref{bounda}) 
(the latter being valid for $p\in [1,1+\alpha ]$), there exists $M_T >0$
independent of $X^\ast$ such that 
$N\stackrel{\text{\tiny def}}{=}\|a\|_{L^\infty([0,T],\ell_p)}\leq M_T$. 

With the above remarks one can apply 
theorem \ref{th:classical} for $\eps\leq\eps_T(M_T)=\tilde\eps_T$.
Noticing that $X_A(t)=X_a^\eps(t)$, 
one obtains estimate (\ref{erreur}) with $\tilde{C}_T=C_T(M_T)$,
which proves theorem \ref{th:classicalbis}. 
\end{proof}

\subsubsection{\label{thmO(eps(1-alpha)ln(1/eps))}
Asymptotics for times $\mathcal{O}(|\ln\eps|\, \eps^{1-\alpha })$}

Theorem \ref{th:nonclassical} below provides
a different kind of error estimate, where
the multiple-scale approximation is controlled on longer 
$\mathcal{O}(|\ln\eps|\, \eps^{1-\alpha })$ time scales, at the expense of
lowering the precision of (\ref{erroran}). 
These estimates are valid
when the Ansatz $X_a^\eps (t) \in \ell_p^2$ is bounded for $t\in \mathbb{R}^+$,
i.e. $X_a^\eps$ is constructed from a solution $a\in L^\infty (\bbR^+,\ell_p)$
of \eqref{eq:ampl}.
The proof of this result is detailed in section \ref{nonclassical}.

\begin{Theorem}
\label{th:nonclassical}
Let $\eta =\min(\alpha-1,\beta,\gamma)$, $\mu \in (0,\eta )$ and $C_{\rm{i}} >0$.
Fix $a\in L^\infty (\bbR^+,\ell_p)$ solution to equation \eqref{eq:ampl} 
with $\|a\|_{L^\infty(\bbR^+,\ell_p)}=N$ and consider 
\begin{equation*}
X_a^\eps(t)= \frac\eps{\sqrt2} a(\eps^{\alpha-1}t)e^{it}\begin{pmatrix}1\\i\end{pmatrix} + \cc
\end{equation*}
There exist positive constants
$\eps_0(\mu , C_{\rm{i}} ,N)$, $C_{\rm{l}}(C_{\rm{i}} ,N)$ and a c.n.i.f. $\nu (N)$
such that for all $\eps\leq\eps_0$, 
if $X^0\in\ell_p^2$ satisfies $\|X^0-X_a^\eps(0)\|_p \leq C_{\rm{i}}\, |\ln\eps|\eps^{1+\eta}$, 
then the solution $X(t)\in \ell_p^2$ to equation \eqref{eq:1storder} with $X(0)=X^0$ is defined for 
$t\in[0,\mu\, \nu \, |\ln\eps|\, \eps^{1-\alpha}]$ and satisfies
\begin{equation}
\label{estthmln}
\|X(t)-X_a^\eps(t)\|_p \leq C_{\rm{l}}\, 
|\ln\eps|\, \eps^{1+\eta - \mu}, \quad 
t\in\left[0,\mu\, \nu \, |\ln\eps|\, \eps^{1-\alpha}\right].
\end{equation}
\end{Theorem}

In the same way as
theorem \ref{th:classicalbis} was deduced from theorem \ref{th:classical}, theorem \ref{th:nonclassicalbis} below 
follows directly from theorem \ref{th:nonclassical}.
The proof requires all solutions to \eqref{eq:ampl} to be
global and bounded in $\ell_p$ (due to the same assumption made on $a$
in theorem \ref{th:nonclassical}), hence we have to restrict to $p=2$.

\begin{Theorem}
\label{th:nonclassicalbis}
Let $\eta =\min(\alpha-1,\beta,\gamma)$, $\mu \in (0,\eta )$ and $C_{\rm{i}} >0$.
There exist positive constants
$\eps_0(\mu , C_{\rm{i}} )$, $C_{\rm{l}}(C_{\rm{i}} )$ and $\nu$
such that the following holds.
For any $X^\ast = (x^\ast, y^\ast)^T\in \ell_2^2$ such 
that $\eps \stackrel{\text{\tiny def}}{=}  \| X^\ast \|_2\leq \eps_0$,
we consider the solution $A\in \calC^2 (\mathbb{R},\ell_2)$ to equation \eqref{eq:ampl}
with $\mathcal{O}(\eps )$ initial condition 
$A (0)= (x^\ast-i \, y^\ast)/\sqrt2$,
and we define 
\begin{equation*}
X_A(t)= \frac{1}{\sqrt2} A(t)\, e^{it}\begin{pmatrix}1\\i\end{pmatrix} + \cc 
\end{equation*}
Then, for all $X^0 \in \ell_2^2$ satisfying $\|X^0-X^\ast\|_2 \leq C_{\rm{i}}\, |\ln\eps|\eps^{1+\eta}$,
the solution $X(t)\in \ell_p^2$ to equation \eqref{eq:1storder} with $X(0)=X^0$ is defined for 
$t\in[0,\mu\, \nu \, |\ln\eps|\, \eps^{1-\alpha}]$ and satisfies
\begin{equation}
\label{erreurln}
\|X(t)-X_A(t)\|_2 \leq C_{\rm{l}}\, 
|\ln\eps|\, \eps^{1+\eta - \mu}, \quad 
t\in\left[0,\mu\, \nu \, |\ln\eps|\, \eps^{1-\alpha}\right] .
\end{equation}
\end{Theorem}

\begin{proof}
We proceed exactly as in the proof of theorem \ref{th:classicalbis} for $p=2$, except the solutions
$A,a$ of \eqref{eq:ampl} are now global in time, 
$N\stackrel{\text{\tiny def}}{=}\|a\|_{L^\infty(\mathbb{R}^+,\ell_2)}=\| a(0) \|_2 = \sqrt{2}/2 $,
and we use theorem \ref{th:nonclassical} instead of theorem \ref{th:classical}.
\end{proof}

\subsubsection{\label{applithm}Long-lived localized solutions}

We can apply theorem \ref{homoclinic} to generate breather 
solutions to the amplitude DpS equation \eqref{dpsa} which provide 
approximate solutions for theorems \ref{th:classical}, \ref{th:classicalbis}, 
\ref{th:nonclassical} and \ref{th:nonclassicalbis}. Hence, we obtain stable
exact solutions to the original nonlinear lattice
\eqref{eq:cradle}, close to breathers, over the corresponding time scales. 
 
\begin{Theorem}
Let $\eta=\min(\alpha-1,\beta,\gamma)$ and fix two constants $C_{\rm{i}},T>0$.
Consider a solution $v^i=(v_n^i)_{n\in \mathbb{Z}}$ ($i=1,2$) of 
the stationary DpS equation (\ref{dpsstat}) described in theorem
\ref{homoclinic}. There exist $\eps_T , C_T >0$ such that for all
$\eps \in (0, \eps_T ]$, for all $X^0 \in \ell_p^2$ satisfying 
\begin{equation}
\label{condic}
\|X^0- (\sqrt{2} \eps v^i ,0)^T \|_p \leq C_{\rm{i}}\eps^{1+\eta},
\end{equation}
the solution $X(t)\in \ell_p^2$ to 
equation \eqref{eq:1storder} with $X(0)=X^0$ 
is defined at least for $t\in[0,T\eps^{1-\alpha}]$ and satisfies
\begin{equation}
\label{erreurb}
\|X(t)-X^\eps_{\rm{b}}(t)\|_p \leq {C}_T\eps^{1+\eta} 
\textrm{ for all }t\in[0,T\eps^{1-\alpha}],
\end{equation}
where
$$
X^\eps_{\rm{b}}(t) = \sqrt{2} \, \eps \, ( v^i\, \cos{(\Omega\, t)},- v^i\, \sin{(\Omega\, t)})^T,
\ \ \
\Omega = 1+\omega_0\, \eps^{\alpha -1}.
$$
\end{Theorem}

\begin{proof}
Consider the breather solution of 
(\ref{dpsa}) given by (\ref{solper}) with $v=v^i$, $\eps=1$,
and apply theorem \ref{th:classical}.
\end{proof}

As a result of estimate (\ref{erreurb}), the initial condition
$X(0)=(\sqrt{2} \eps v^i ,0)^T$ generates long-lived breather solutions
$\tilde{X}^\eps_{\rm{b}}$ defined for $t\in [0, T\, \eps^{1-\alpha}]$
and taking the form $\tilde{X}^\eps_{\rm{b}}(t)={X}^\eps_{\rm{b}}(t)+\mathcal{O}(\eps^{1+\eta})$.
These solutions are stable in $\ell^2_p$ on the corresponding time scale
since condition (\ref{condic}) implies
$$
\|X(t)-\tilde{X}^\eps_{\rm{b}}(t)\|_p \leq 2\, {C}_T\eps^{1+\eta} 
\textrm{ for all }t\in[0,T\eps^{1-\alpha}]
$$
(this follows by using (\ref{erreurb}) and the triangle inequality).

\vspace{1ex}

Using theorem \ref{th:nonclassicalbis} and corollary \ref{colnd}, we also
obtain lower bounds for the amplitudes of small localized solutions over long times.
This result is valid for all initial data in subsets $\mathcal{C}$ of $\ell_2^2$
such that $\mathcal{I}(\mathcal{C})>0$, where $\mathcal{I}(\mathcal{C})$
is defined as in (\ref{defi})-(\ref{defq}) with the choice of norms (\ref{normelp})
(for which the canonical isomorphism between 
${(\ell_p(\mathbb{Z},\mathbb{R}) )}^2$ and $\ell_p(\mathbb{Z},\mathbb{C})$
is an isometry).
As already noticed in remark \ref{remi}, one has $\mathcal{I}(\mathcal{C})>0$ whenever $\mathcal{C}$
is a finite-dimensional linear subspace of $\ell_2^2$.

\begin{Proposition} \label{prop:nodispers}
Keep the notations of theorem \ref{th:nonclassicalbis}.
Let $\mathcal{C}$ denote a subset of $\ell_2^2$ such that
$\mathcal{I}(\mathcal{C})>0$.  
There exists $\eps_1(\mu,C_{\rm{i}},\mathcal{I}(\mathcal{C}))>0$ such that
for all $\eps\in(0,\eps_1]$ and $X^0 \in \mathcal{C}$ 
with $\| X^0 \|_2 = \eps$, the solution $X$ 
to equation \eqref{eq:1storder} given by 
theorem \ref{th:nonclassicalbis} is bounded from below, 
namely:    
\begin{equation}
\label{threedots}
\forall t\in\left[0,\mu\, \nu \, |\ln\eps|\, \eps^{1-\alpha}\right], 
\quad \|X(t)\|_\infty \geq M\, \eps
\end{equation}
with 
$
M=\frac{1}{2}\, {\left(
\frac{1}{2}\, \mathcal{I}(\mathcal{C})
\right)}^\frac{\alpha +1}{\alpha-1}
$.
\end{Proposition} 

\begin{proof}
Fix $X^0 = X^\ast$ in theorem \ref{th:nonclassicalbis} and
note that 
\begin{eqnarray*}
\|X(t)\|_\infty & \geq & \|X_A(t)\|_\infty - \|X(t)-X_A(t)\|_\infty \\
& \geq & \|X_A(t)\|_\infty - \|X(t)-X_A(t)\|_2 . 
\end{eqnarray*}
Now, estimating $\|X_A(t)\|_\infty$ thanks to corollary \ref{colnd}
and $\|X(t)-X_A(t)\|_2$ with theorem \ref{th:nonclassicalbis} 
gives $ \|X(t)\|_\infty \geq \eps\, (2M- C_{\rm{l}}\, |\ln\eps|\, \eps^{\eta - \mu}) $.
Since $\eta > \mu$, this estimate implies
(\ref{threedots}) provided $\eps$ is small enough.
\end{proof}

\begin{Example}
Consider solutions $x=(x_n)_{n\in\mathbb{Z}}$ to \eqref{eq:cradle} 
with unperturbed initial positions, and with a group of 
$N$ consecutive particles having the same initial velocity $v_{\rm{i}}$ ($N\geq 1$ being fixed).
This corresponds to fixing 
$$
\forall n\in\mathbb{Z}, \,\, x_n(0)=0 \, , \quad 
\dot x_n(0) = 
\left\{ 
\begin{array}{ll}
{v_{\rm{i}}} & \text{ if } n\in\{1,\dots,N\}, \\
0 & \text{ elsewhere},
\end{array}
\right.
$$ 
i.e. $X^0 = v_{\rm{i}}\, (0, \mathbbm{1}_{\{ 1,\ldots ,N \}})^T$.
For $\mathcal{C}=\mbox{Span}\left( \, (0, \mathbbm{1}_{\{ 1,\ldots ,N \}})^T  \, \right)$
one has $\mathcal{I}( \mathcal{C} )
=Q_\alpha ( \mathbbm{1}_{\{ 1,\ldots ,N \}}) = 2^\frac{1}{\alpha+1} N^{-1/2}>0$
(see definition (\ref{defq})). Consequently one can apply proposition 
\ref{prop:nodispers}, where $M=2^\frac{2\alpha -1}{1-\alpha } N^\frac{\alpha +1}{2(1-\alpha )}$
and $\| X^0 \|_2 =|v_{\rm{i}}|\, N^{1/2}= \eps$. This yields for all $\eps \in (0,\eps_1 ]$ and
$t\in\left[0,\mu\, \nu \, |\ln\eps|\, \eps^{1-\alpha}\right]$~:
$$
\sup_{n\in\bbZ}{(x_n^2(t)+\dot{x}_n^2(t))^{1/2}} 
\geq 2^\frac{2\alpha -1}{1-\alpha }\, N^\frac{1}{1-\alpha }\, |v_{\rm{i}}|.
$$
\end{Example}

To interpret estimate (\ref{threedots}), it is interesting to recall that
$\|X(t)\|_\infty$ is conserved along evolution for solutions to 
the linearized equation (\ref{eq:cradlelin2})
(see remark \ref{remcons}), hence the bound (\ref{threedots}) estimates
the maximal decay of $\|X \|_\infty$ that could occur
over long times due to purely nonlinear effects.
This estimate can be compared with the classical Gronwall estimate given below.

\begin{Lemma}
Keep the notations of lemma \ref{cauchypb}. There exists a constant 
$\tilde\eps_1  >0$ such that the following property holds true.
For all $\delta \in (0,1)$, there exists $\tilde{T}(\delta ) \in (0,T_0 ]$
such that for all $X^0 \in \ell_\infty^2$ with $\| X^0 \|_\infty \leq \tilde\eps_1$ one has
\begin{equation}
\label{gw}
\|X(t)\|_\infty \geq \delta\, \| X^0 \|_\infty,
\ \ \ 
\forall\, t \in [0, \tilde{T}(\delta ) \, \| X^0 \|_\infty^{1-\alpha} ].
\end{equation}
\end{Lemma}

\begin{proof}
Using the triangle inequality in the
Duhamel form (\ref{duha}), the time-invariance of $\| e^{Jt} X^0\|_\infty $ and
estimate (\ref{eg}), one finds
$$
\|X(t)\|_\infty \geq \|X^0\|_\infty - \lambda
\int_0^t \|X(s)\|_\infty^\alpha\ \rmd s .
$$
Then we deduce from the Gronwall estimate (\ref{estimx})
$$
\|X(t)\|_\infty \geq \|X^0\|_\infty - t\, \lambda \, \theta^{\frac{\alpha}{1-\alpha }}\, \|X^0\|_\infty^\alpha ,
$$
from which the result follows easily.
\end{proof}

Estimate (\ref{gw}) differs from (\ref{threedots}) in the sense that it involves
only the $\ell_\infty$
norm and holds true for all sufficiently small initial data in $ \ell_\infty^2$.
In addition it is valid on time scales of order $ \| X^0 \|_\infty^{1-\alpha}$, whereas
(\ref{threedots}) holds true on longer time scales of order
$|\ln{(\| X^0 \|_2)}|\, \| X^0 \|_2^{1-\alpha}$.

\section{\label{just}Error bounds via Gronwall estimates}

In this section we prove theorems \ref{th:classical}
and \ref{th:nonclassical}. For this purpose, we check
in section \ref{resi}
the consistency of the Ansatz $X_\rmapp^\eps$ 
defined by (\ref{approxsolu})
and conclude using Gronwall estimates
(sections \ref{errorclassic} and \ref{nonclassical}).

\subsection{\label{resi}Estimate of the residual}

In the sequel we consider solutions $a(\tau )$ to equation \eqref{dpsa}
such that $a\in L^\infty(I,\ell_p)$ for some closed time interval $I$.
By lemma \ref{lm:estim_a}, this can be achieved for all initial condition
$a(0)=a^0\in\ell_p(\bbZ,\bbC)$ by choosing $I=[0,T]\subset (T_{\rm{min}},T_{\rm{max}})$
in the general case,
or $I=[0, +\infty )$ in the particular case $p=2$.
The amplitude $a$ determines
the approximate solution $X_\rmapp^\eps$ to equation \eqref{eq:1storder}
introduced in section \ref{multiple}. Following equation (\ref{approxsolu}),
we recall that
$$
X_\rmapp^\eps(t) = \eps Y^\eps (\tau ,t), \quad \tau = \eps^{\alpha-1}t,
$$
where 
$Y^\eps ={Y}_0 + \eps^{\alpha -1} {Y}_1$ and
$Y_k (\tau ,t)=(\mathcal{Y}_k(\tau))(t)$ for $k=0,1$.

In this section we check that $X_\rmapp^\eps$ solves
equation \eqref{eq:1storder} up to the residual
\begin{equation}
\label{defeeps}
E^\eps=\dot X_\rmapp^\eps - J\, X_\rmapp^\eps - G(X_\rmapp^\eps)
\end{equation}
that remains
$o(\varepsilon^{\alpha})$ when $\eps^{\alpha-1}t\in I$.

To this aim, we first prove the following lemma providing
bounds on the approximate solutions
$\mathcal{Y}_0$ and the correctors $\mathcal{Y}_1$ derived from
the amplitudes $a(\tau )$.
Below we denote by $\calC^k_b(I,\bbD)$ the Banach space of
$\calC^k$ functions from $I$ into $\bbD$ with bounded derivatives up to order $k$, 
equipped with the usual supremum norm
(see section \ref{sec:linear} for the definition of function spaces $\mathbb{X}^k$).

\begin{Lemma}
\label{lm:estim_tY}
Consider any solution $a\in L^\infty(I,\ell_p)$
to equation \eqref{dpsa},
the associated leading order solution $\mathcal{Y}_0$ of (\ref{eq:2scales2}) defined by
(\ref{y0}), and its corrector $\mathcal{Y}_1$ defined by (\ref{o1sol}). 
There exist $M_0, M_1 >0$ such that
\begin{equation}
\label{estlem}
\| \mathcal{Y}_0 \|_{\calC^1_b(I,\bbD)} 
 \leq   M_0\, (\|a\|_{L^\infty(I,\ell_p)}+\|a\|_{L^\infty(I,\ell_p)}^\alpha ),
\end{equation}
\begin{equation}
\label{estlem2}
\| \mathcal{Y}_1 \|_{\calC^1_b(I,\bbD)} 
 \leq   M_1\, (\|a\|_{L^\infty(I,\ell_p)}^{\alpha} + \|a\|_{L^\infty(I,\ell_p)}^{2\alpha -1}).
\end{equation}
\end{Lemma}

\begin{proof}
We immediately deduce from Lemma \ref{lm:estim_a} that 
$\mathcal{Y}_0\in\calC^2_b(I,\bbD)$ and
\begin{equation}
\label{y0inf}
\| \mathcal{Y}_0 \|_{L^\infty (I,\mathbb{X}^k)}  \leq  (1+k)\, \sqrt{2}\, \|a\|_{L^\infty(I,\ell_p)},
\quad k=0,1.
\end{equation}
In what follows we estimate $G_\alpha(\mathcal{Y}_0)$, in order to 
estimate $\partial_\tau \mathcal{Y}_0$ from equation (\ref{ampli}) and
$\mathcal{Y}_1$ from equation (\ref{o1sol}). 
Since $G_\alpha \in \calC^1 (\ell_p^2,\ell_p^2)$ we have also $G_\alpha \in \calC^1 (\bbX , \bbX)$
by the omega-lemma (see \cite{amr}, lemma 2.4.18). Moreover, standard estimates yield for all $X\in \ell_p^2$
\begin{equation}
\label{estimgal}
\| G_\alpha (X) \|_p \leq M_3\,  \|X\|_p^{\alpha}, \quad
\| DG_\alpha (X) \|_{\mathcal{L}(\ell_p^2)} \leq M_4 \, \|X\|_p^{\alpha -1},
\end{equation}
which implies for all $X\in \bbX$
$$
\| G_\alpha (X) \|_{\bbX} \leq M_3 \, \|X\|_{\bbX}^{\alpha}, \quad
\| DG_\alpha (X) \|_{\mathcal{L}({\bbX})} \leq M_4 \, \|X\|_{\bbX}^{\alpha -1}.
$$
We have then
\begin{equation*}
\|G_\alpha(\mathcal{Y}_0)\|_{L^\infty(I,\bbX)} 
\leq M_3 \, \|\mathcal{Y}_0\|_{L^\infty(I,\bbX)}^{\alpha},
\end{equation*}
\begin{equation*}
\|\partial_\tau \mathcal{Y}_0\|_{L^\infty(I,\bbX)}
\leq \| P \|_{\mathcal{L}(\mathbb{X}^0)}\, \|G_\alpha(\mathcal{Y}_0)\|_{L^\infty(I,\bbX)} \leq
M_3 \, \|\mathcal{Y}_0\|_{L^\infty(I,\bbX)}^{\alpha},
\end{equation*}
\begin{equation*}
\|\partial_\tau G_\alpha(\mathcal{Y}_0)\|_{L^\infty(I,\bbX)} \leq
M_4\, \|\mathcal{Y}_0\|_{L^\infty(I,\bbX)}^{\alpha -1} \, \|\partial_\tau \mathcal{Y}_0\|_{L^\infty(I,\bbX)}
\leq M_3 M_4\, \|\mathcal{Y}_0\|_{L^\infty(I,\bbX)}^{2\alpha -1}
\end{equation*}
(the estimate of $\partial_\tau \mathcal{Y}_0$ follows from 
equation (\ref{ampli}) and $\| P \|_{\mathcal{L}(\mathbb{X}^0)}=1$).
From these estimates and (\ref{y0inf}), 
there exists $ M_5 >0$ such that
\begin{equation}
\label{dy0}
\|\partial_\tau \mathcal{Y}_0\|_{L^\infty(I,\mathbb{X}^1)}
\leq  M_5\, \| a\|_{L^\infty(I,\ell_p)}^{\alpha}, 
\end{equation}
\begin{equation}
\label{g0}
\|G_\alpha(\mathcal{Y}_0)\|_{L^\infty(I,\bbX)} 
\leq  M_5\, \| a\|_{L^\infty(I,\ell_p)}^{\alpha}, 
\end{equation}
\begin{equation}
\label{g1}
\|\partial_\tau G_\alpha(\mathcal{Y}_0)\|_{L^\infty(I,\bbX)} \leq
 M_5\, \|a\|_{L^\infty(I,\ell_p)}^{2\alpha -1}.
\end{equation}
Consequently, estimate (\ref{estlem}) is established thanks to (\ref{y0inf})
and (\ref{dy0}).
Furthermore, using the fact that ${\mathcal{K}}\in \mathcal{L}(\bbX,\bbD)$, 
we obtain $\mathcal{Y}_1 = {\mathcal{K}}G_\alpha(\mathcal{Y}_0)\in \calC^1_b(I,\bbD)$ and
\begin{equation*}
\|\mathcal{Y}_1\|_{\calC^1_b(I,\bbD)} \leq 
\| \mathcal{K} \|_{\mathcal{L}(\bbX,\bbD)}  M_5\, (\|a\|_{L^\infty(I,\ell_p)}^{\alpha} + \|a\|_{L^\infty(I,\ell_p)}^{2\alpha -1}),
\end{equation*}
which establishes
estimate (\ref{estlem2}).
\end{proof}

Now we prove the main result of this section.
The subsequent estimates will involve 
c.n.d.f. of various norms which we will denote by $\CNDF_k$.

\begin{Lemma}
\label{lm:residual}
There exist a c.n.i.f. $\CNIF_1$ and a c.n.d.f. $\CNDF_{\rm{E}}$
such that for all $a\in L^\infty(I,\ell_p)$ 
solution to equation \eqref{dpsa}
and $\eps\leq\CNIF_1(\|a\|_{L^\infty(I,\ell_p)})$, 
the residual $E^\eps$ defined by (\ref{defeeps}) satisfies
\begin{equation}
\label{ineqlemma}
\sup_{\eps^{\alpha-1}t\in I} \|E^\eps(t)\|_p 
\leq \CNDF_{\rm{E}}(\|a\|_{L^\infty(I,\ell_p)}) \eps^{\alpha + \eta},
\end{equation}
where $\eta=\min(\alpha-1,\beta,\gamma)$.
\end{Lemma}

\begin{proof}
Let us compute $E^\eps$. Using identities (\ref{approxsolu}),
(\ref{ordre0}), (\ref{o1}) and (\ref{decompg}), one obtains after some elementary computations
\begin{eqnarray}
\label{decompres}
E^\eps (t)
&= &  \eps^{\alpha} [G_\alpha({Y}_0(\eps^{\alpha-1} t,t)) - G_\alpha(Y^\eps(\eps^{\alpha-1} t,t) )] \\
\nonumber
&+& \eps^{2\alpha-1} \partial_\tau {Y}_1(\eps^{\alpha-1}t ,t) - \eps^{\alpha} R_\eps(Y^\eps (\eps^{\alpha-1}t ,t)).
\end{eqnarray}
Let us estimate each term of (\ref{decompres}) separately. 
From lemma \ref{lm:estim_tY} we already know 
a c.n.d.f. $\CNDF_1$ such that
\begin{equation}
\label{estdt}
\sup_{\eps^{\alpha-1}t\in I}
\|\partial_\tau Y_1(\eps^{\alpha-1}t,t)\|_p \leq \CNDF_1(\|a\|_{L^\infty(I,\ell_p)}). 
\end{equation}
Moreover, by estimate (\ref{estre}) and lemma \ref{lm:estim_tY},
there exists a c.n.i.f. $\CNIF_1$ and a c.n.d.f. $\CNDF_2$ such that 
for $\eps\leq\CNIF_1(\|a\|_{L^\infty(I,\ell_p)})$
\begin{equation}
\label{estrest}
\sup_{\eps^{\alpha-1}t\in I} \|R_\eps( Y^\eps (\eps^{\alpha-1}t ,t))\|_p 
\leq \eps^{\min(\beta,\gamma)} \CNDF_2(\|a\|_{L^\infty(I,\ell_p)}).
\end{equation}
Let us impose $\CNIF_1 \leq 1$ without loss of generality.
The first term at the right side
of (\ref{decompres}) can be estimated as follows for $\eps^{\alpha-1}t\in I$,
using estimate (\ref{estimgal}) and lemma \ref{lm:estim_tY}~:
\begin{equation*}
\begin{aligned}
\| & G_\alpha({Y}_0(\eps^{\alpha-1} t,t)) - G_\alpha(Y^\eps(\eps^{\alpha-1} t,t) ) \|_p \\
& = \eps^{\alpha-1}\big\|\int_0^1 DG_\alpha(Y_0(\eps^{\alpha-1} t,t)+\theta\eps^{\alpha-1}Y_1(\eps^{\alpha-1} t,t))\ \rmd\theta\cdot 
Y_1(\eps^{\alpha-1} t,t)\big\|_p \\
& \leq \eps^{\alpha-1}\int_0^1 \|DG_\alpha(Y_0(\eps^{\alpha-1} t,t)+\theta\eps^{\alpha-1}Y_1(\eps^{\alpha-1} t,t))\|_{\calL(\ell_p^2)} 
 \|Y_1(\eps^{\alpha-1} t,t)\|_p \\
& \leq \eps^{\alpha-1}\CNDF_3(\|Y_0(\eps^{\alpha-1} t,t)\|_p+ \|Y_1(\eps^{\alpha-1} t,t)\|_p)\
\|Y_1(\eps^{\alpha-1} t,t)\|_p \\
& \leq \eps^{\alpha-1}\CNDF_4(\|a\|_{L^\infty(I,\ell_p)}).
\end{aligned}
\end{equation*}
Combining this estimate with (\ref{estdt}) and (\ref{estrest})
yields the final estimate (\ref{ineqlemma}).
\end{proof}

\vspace{1ex}

\begin{Example}
In the case $\alpha = 3/2$ (as for the classical Hertz force), if $W$ and $\phi$ are
$\mathcal{C}^3$ at both sides of the origin, then $\beta , \gamma \geq 1/2$
in (\ref{as:w}) and (\ref{as:phi}),
and we have by lemma \ref{lm:residual}
$$
\sup_{\eps^{1/2}t\in I} \|E^\eps(t)\|_p 
\leq \CNDF_{\rm{E}}(\|a\|_{L^\infty(I,\ell_p)}) \eps^{2}.
$$
\end{Example}

\subsection{\label{errorclassic}Proof of theorem \ref{th:classical}} 

To prove theorem \ref{th:classical},
we first estimate the error between
the approximate solution $X_\rmapp^\eps$ 
constructed from a given $a(0)=a^0\in\ell_p(\bbZ,\bbC)$
(equation (\ref{approxsolu})) and the exact solution 
$X$ to equation \eqref{eq:1storder} for $X(0)\approx X_\rmapp^\eps (0)$,
in the case when $\eps \approx 0$ and on $\mathcal{O}(\eps^{1-\alpha})$
time scales. This result will follow in a rather standard way from
Gronwall estimates.
In a second step, we check on these
time scales the validity of the
leading-order approximate solution $X_{a}^\eps$
(equation (\ref{approxsoluo1})).

\vspace{1ex}

By lemma \ref{lm:estim_a}, for all initial condition $a^0\in\ell_p(\bbZ,\bbC)$
equation \eqref{dpsa} admits a unique maximal solution $a$ defined 
for $\tau \in (T_{\rm{min}},T_{\rm{max}})$. Let us fix $T\in (0,T_{\rm{max}})$
and restrict $a$ to $[0,T]$, so that $a\in L^\infty([0,T],\ell_p)$. 
From lemma \ref{lm:estim_tY} it follows  
that $X_\rmapp^\eps \in L^\infty([0,t_\rmm(\eps)],\ell_p^2)$
with $t_\rmm(\eps)=T/\eps^{\alpha-1}$, and we have $E^\eps \in L^\infty([0,t_\rmm(\eps)],\ell_p^2)$
by lemma \ref{lm:residual}.

Let $Z=X-X_\rmapp^\eps$ and $Z_0 \equiv X(0)-X_\rmapp^\eps(0)$.
Then $Z$ is solution to 
\begin{eqnarray*}
\dot Z - J\, Z & = & G(X_\rmapp^\eps+Z)-G(X_\rmapp^\eps) - E^\eps, 
\\
Z(0) & = & Z_0 ,
\end{eqnarray*}
or equivalently 
\begin{equation}
\label{duhaz}
Z(t) = e^{Jt} Z_0 + \int_0^t e^{J(t-s)} [G(X_\rmapp^\eps+Z)-G(X_\rmapp^\eps)-E^\eps](s)\ \rmd s.
\end{equation}
By Cauchy--Lipschitz theorem, the solution 
$Z\in \calC^1([0,t_\rmmax (\eps)],\ell_p^2)$ 
to this equation is 
defined up to some maximal existence time $t_\rmmax (\eps ) \leq t_\rmm(\eps)$, depending {\it a priori} on $Z_0$. 

\begin{Remark}
In fact we prove below that
$t_\rmmax (\eps) = t_\rmm(\eps)$ when $\eps$ and $\| Z_0 \|_p $ are small enough.
This will provide a solution $X$ to \eqref{eq:1storder} defined 
at least for $t\in [0, t_\rmm(\eps) ]$.
\end{Remark}

We have already estimated the residual $E^\eps$ in the previous section.
We now need to estimate the difference  $G(X_\rmapp^\eps+Z)-G(X_\rmapp^\eps)$. 
For this purpose we assume $\|Z_0\|_p < \eps$ and define 
$$
t_\eps ( Z_0)=\sup\{t\in[0,t_\rmmax], \ \|Z(t)\|_p \leq \eps\}.
$$
\begin{Remark}
\label{remte}
Since $\|Z \|_p \in \calC^0 ([0,t_\rmmax])$, we have either 
$t_\eps = t_\rmmax$ or $t_\eps < t_\rmmax$ and
$\|Z(t_\eps )\|_p = \eps $.
\end{Remark}

\begin{Lemma}
\label{lm:Lipschitz}
There exists a c.n.i.f. $\eps_2$ and a c.n.d.f. $\CNDF_{\rm{L}}$ such that 
for all solution $a\in L^\infty([0,T],\ell_p)$ to equation \eqref{dpsa}, 
$\eps\leq\CNIF_2(\|a\|_{L^\infty([0,T],\ell_p)})$, $Z_0 \in \ell_p^2$ with
$\|Z_0\|_p < \eps$ and $t\in[0,t_\eps ( Z_0)]$, 
\begin{equation}
\label{estlemlip}
\|G(X_\rmapp^\eps(t)+Z(t))-G(X_\rmapp^\eps(t))\|_p 
\leq \eps^{\alpha-1} \CNDF_{\rm{L}}(\|a\|_{L^\infty([0,T],\ell_p)})\,  \|Z(t)\|_p.
\end{equation}
\end{Lemma}

\begin{proof}
We first estimate
\begin{equation}
\label{accf}
\|G(X_\rmapp^\eps(t)+Z(t))-G(X_\rmapp^\eps(t))\|_p 
\leq \int_0^1 \|DG(X_\rmapp^\eps(t)+\theta Z(t))\|_{\calL(\ell_p^2)}\rmd \theta\  \|Z(t)\|_p.
\end{equation}
By estimate (\ref{estmapg}), there exists $\mu , C>0$ such that
for all $U\in \ell_p^2$ with $\| U \|_p \leq \mu$, we have
\begin{equation}
\label{eq:estim_U}
\| DG(U)\|_{\mathcal{L}(\ell_p^2)} \leq C\, \| U \|_p^{\alpha-1}.
\end{equation}
From definition (\ref{approxsolu}) and lemma \ref{lm:estim_tY},
assuming $\CNIF_2 \leq 1$,
there exists a c.n.d.f. $\CNDF_5$ such that 
\begin{equation*}
\sup_{t\in[0,t_\rmm(\eps)]} \|X_\rmapp^\eps(t)\|_p \leq \eps\, \CNDF_5(\|a\|_{L^\infty([0,T],\ell_p)}) 
\quad \text{for all } \eps\in(0,\CNIF_2],
\end{equation*}
hence for all $\theta \in [0,1]$ and $t\in[0,t_\eps (Z_0)]$
\begin{equation}
\label{estxpz}
\|X_\rmapp^\eps(t)+\theta Z(t)\|_p 
\leq \eps (1+\CNDF_5(\|a\|_{L^\infty([0,T],\ell_p)})).
\end{equation}
Consequently, 
there exists a c.n.i.f. $\CNIF_2 \leq 1$ such that for all $\eps\leq\CNIF_2(\|a\|_{L^\infty([0,T],\ell_p)})$,
$\theta \in [0,1]$ and $t\in[0,t_\eps ( Z_0)]$ we have
$\|X_\rmapp^\eps(t)+\theta Z(t)\|_p \leq \mu$. Then one obtains
estimate (\ref{estlemlip}) using 
bounds (\ref{eq:estim_U}) and (\ref{estxpz}) in conjunction with estimate (\ref{accf}).
\end{proof}

\vspace{1ex}

Now let us apply lemmas \ref{lm:residual} and \ref{lm:Lipschitz} to the Duhamel formulation 
(\ref{duhaz}) for $Z(t)$. 
This yields for all $\eps\leq\CNIF_2(\|a\|_{L^\infty([0,T],\ell_p)})$, 
$Z_0 \in \ell_p^2$ with $\|Z_0\|_p < \eps$
and for all $t\in [0 , t_\eps ]$
\begin{equation*}
\|Z(t)\|_p \leq b_\eps + C_{\rm{L}} \eps^{\alpha-1} \int_0^t \|Z(s)\|_p\ \rmd s,
\end{equation*}
with $b_\eps = \|Z_0\|_p + C_{\rm{E}} T \eps^{1+\eta}$.
Then one obtains by Gronwall lemma
\begin{equation}
\label{bornez}
\|Z(t)\|_p \leq b_\eps \exp(C_{\rm{L}} \eps^{\alpha-1}t_\eps) \leq b_\eps \exp(C_{\rm{L}} T).
\end{equation}

We now fix a c.n.d.f. $\CNDF_0 >0$ and
make the following stronger assumption on the initial distance between
exact and approximate solutions. 

\begin{Assumption} 
\label{as:Z0}
$Z_0=X(0)-X_\rmapp^\eps(0)$ satisfies
$$\|Z_0\|_p \leq \CNDF_0 (\|a\|_{L^\infty([0,T],\ell_p)}) \, T\, \eps^{1+\eta}.$$ 
\end{Assumption}

Assumption \ref{as:Z0} implies $\|Z_0\|_p <\eps $ as soon as
$\eps < \eps_3 =(C_0\ T)^{-1/\eta}$, and then estimate (\ref{bornez}) applies
for $\eps <  \rm{min}(\eps_2 , \eps_3)$. This yields
\begin{equation}
\label{estzsmall}
\|Z(t)\|_p \leq \CNDF_{\rm{R}} (\|a\|_{L^\infty([0,T],\ell_p)}) \,  \eps^{1+\eta}, \quad
t\in [0 , t_\eps ],
\end{equation}
where $\CNDF_{\rm{R}} = (C_0 + C_{\rm{E}})\, T\, \exp(C_{\rm{L}} T)$ is a c.n.d.f.
Consequently, for $\eps < \eps_T = \rm{min}(\eps_2 , C_{\rm{R}}^{-1/\eta})$ we have
$\|Z(t)\|_p < \eps$ for all $t\in [0 , t_\eps ]$, which implies
$t_\eps = t_\rmmax$ (see remark \ref{remte}).
Now estimate (\ref{estzsmall}) allows to bound $\|Z(t)\|_p$
for $t\in [0 , t_\rmmax ]$. Consequently, for all
$\eps < \eps_T(\|a\|_{L^\infty([0,T],\ell_p)})$ and $Z_0$ satisfying assumption \ref{as:Z0}
we have $t_\rmmax (\eps )= t_\rmm(\eps)$ and
\begin{equation}
\label{estzsmallmax}
\| X(t)-X_\rmapp^\eps (t)  \|_p \leq \CNDF_{\rm{R}} (\|a\|_{L^\infty([0,T],\ell_p)}) \,  \eps^{1+\eta}, \quad
t\in [0 , T/\eps^{\alpha-1} ].
\end{equation}

\vspace{1ex}

With the error estimate (\ref{estzsmallmax}) at hand, one can recover
an estimate of the same type for the leading order approximate solution
$X_a^\eps (t) = \eps Y_0(\eps^{\alpha-1}t,t)$. Indeed
\begin{equation}
\label{diff}
X(t) -  X_a^\eps (t)
= X(t) - X_\rmapp^\eps(t) + \eps^\alpha Y_1(\eps^{\alpha-1}t,t) .
\end{equation}
From lemma \ref{lm:estim_tY}, we already know 
a c.n.d.f. $\CNDF_1$ such that
for all $t\in[0,T/\eps^{\alpha-1}]$,
\begin{equation}
\label{y1eps}
\| Y_1(\eps^{\alpha-1}t,t)\|_p \leq \CNDF_1(\|a\|_{L^\infty([0,T],\ell_p}).
\end{equation}
Now let us set $C_0 = (C_{\rm{i}} +C_1) /T$ in assumption \ref{as:Z0}, where $C_{\rm{i}} >0$ is a fixed constant, and make the assumption 
$\|X^0-X_a^\eps(0)\|_p \leq C_{\rm{i}}\eps^{1+\eta}$
of theorem \ref{th:classical}.
Recalling that $\eps \leq 1$, $\eta +1 \leq \alpha $, 
and using (\ref{diff})--(\ref{y1eps}), one can check that
assumption \ref{as:Z0} is satisfied and thus estimate (\ref{estzsmallmax})
holds. Then using (\ref{diff})--(\ref{y1eps}) again, 
there exists a c.n.d.f. $C_T = C_{\rm{R}} + C_1$ such that for all $t\in[0,T/\eps^{\alpha-1}]$,
\begin{equation*}
\|X(t) - X_a^\eps (t) \|_p \leq
C_T(\|a\|_{L^\infty([0,T],\ell_p})\, \eps^{1+\eta}.
\end{equation*}
This completes the proof of theorem \ref{th:classical}.

\subsection{\label{nonclassical}Proof of theorem \ref{th:nonclassical}} 

Theorem \ref{th:nonclassical} consists of
a different kind of error estimate, where
the multiple-scale approximation is controlled on 
$\mathcal{O}(|\ln\eps|\, \eps^{1-\alpha })$ time scales going beyond $\eps^{1-\alpha}$, 
at the expense of
lowering the precision of (\ref{erroran}). To obtain this result, 
we adapt the proof of theorem \ref{th:classical} by
allowing $T$ to grow logarithmically in $\eps$.
Theorem \ref{th:nonclassical} is valid
when the Ansatz $X_a^\eps (t) \in \ell_p^2$ is bounded for $t\in \mathbb{R}^+$.
In that case, all estimates of sections \ref{resi} and \ref{errorclassic} involving constants
depending only on $\|a\|_{L^\infty([0,T],\ell_p)}$ 
can be made uniform w.r.t. $T$,
which allows one to set $T=\mathcal{O}(|\ln\eps|)$ when $\eps \rightarrow 0$.

\vspace{1ex}

In what follows we use the notations and definitions introduced in
section \ref{errorclassic}.
Let us consider a solution $a \in L^\infty ([0,+\infty ),\ell_p)$ of the amplitude equation \eqref{eq:ampl},
a fixed constant $\mu \in (0,\eta )$ and set
$$
T=\frac{\mu}{\tilde{C}_{\rm{L}}}\, |\ln\eps |,
$$
where $\tilde{C}_{\rm{L}} = C_{\rm{L}} (\|a\|_{L^\infty([0,+\infty ),\ell_p)})$ 
and $\eps \leq 1$. The initial error $Z_0$ is set to satisfy
assumption \ref{as:Z0} with $C_0 =M$, $M$ being a fixed constant
that will be subsequently determined.
It follows that $\| Z_0 \|_p =\mathcal{O}(|\ln\eps |\, \eps^{1+\eta}) < \eps$
provided $\eps$ is small enough. 
Assuming in addition $\eps \leq  \tilde\eps_2 = \eps_2 (\|a\|_{L^\infty([0,+\infty ),\ell_p)})$,
estimate (\ref{estzsmall}) ensures that
\begin{equation}
\label{estzsmallunif}
\|Z(t)\|_p \leq (M + \tilde{C}_{\rm{E}})\, T\, \exp(\tilde{C}_{\rm{L}} T)
 \,  \eps^{1+\eta}, \quad
t\in [0 , t_\eps ],
\end{equation}
where $\tilde{C}_{\rm{E}} = C_{\rm{E}}(\|a\|_{L^\infty([0,+\infty ),\ell_p)})$. 
Our choice of $T$ yields exactly $\exp(\tilde{C}_{\rm{L}} \,T)=\eps^{-\mu}$,
hence estimate (\ref{estzsmallunif}) becomes
\begin{equation}
\label{estzsmallmaxunif}
\|Z(t)\|_p \leq \frac{M + \tilde{C}_{\rm{E}}}{\tilde{C}_{\rm{L}}}\, \mu\, |\ln\eps |\, \eps^{1+\eta - \mu}, \quad
t\in [0 , t_\eps ].
\end{equation}
Consequently, for $\eps$ small enough we have 
$\|Z(t)\|_p < \eps$ for all $t\in [0 , t_\eps ]$, which implies
$t_\eps = t_\rmmax = T\, \eps^{1-\alpha}$ as 
shown in section \ref{errorclassic}.

Now, as previously observed in section \ref{errorclassic},
the error bound (\ref{estzsmallmaxunif}) yields 
an estimate of the same type for 
the leading order approximate solution $X_a^\eps (t)$,
thanks to the estimate 
\begin{equation}
\label{y1epsunif}
\| Y_1(\eps^{\alpha-1}t,t)\|_p \leq \tilde{C}_1, \quad
t\in[0,T/\eps^{\alpha-1}],
\end{equation}
with $\tilde{C}_1 = \CNDF_1(\|a\|_{L^\infty([0,\infty ),\ell_p})$.
Indeed, let us further
assume $\eps \leq e^{-1}$ and make the assumption
$\|X^0-X_a^\eps(0)\|_p \leq C_{\rm{i}}\, |\ln\eps |\, \eps^{1+\eta}$
of theorem \ref{th:nonclassical}
for some fixed constant $C_{\rm{i}}$.
Using identity (\ref{diff}) and the bounds given above, one obtains
$$
\|X^0-X_\rmapp^\eps(0)\|_p \leq (C_{\rm{i}} + \tilde{C}_1)\, |\ln\eps |\, \eps^{1+\eta}
$$
(we recall that $\eta  \leq \alpha -1$).
Consequently, assumption \ref{as:Z0} is satisfied with the choice
$C_0=M= (C_{\rm{i}} +\tilde{C}_1) \tilde{C}_{\rm{L}}/\mu$
and estimate (\ref{estzsmallmaxunif})
holds true for all $t\in[0,T\, \eps^{1-\alpha}]$. 
Then using (\ref{diff})-(\ref{y1epsunif}) and the definition of $T$, we get
for all $t\in[0, {\mu}\, \tilde{C}_{\rm{L}}^{-1}\, |\ln\eps |\, \eps^{1-\alpha}]$
\begin{equation*}
\|X(t) - X_a^\eps (t) \|_p \leq
C_{\rm{l}}\, |\ln\eps|\, \eps^{1+\eta -\mu},
\end{equation*}
with $C_{\rm{l}}=C_{\rm{i}} + 2 \tilde{C}_1 + \eta\, \tilde{C}_{\rm{E}} / \tilde{C}_{\rm{L}}$,
which proves estimate (\ref{estthmln}) for $\nu =  \tilde{C}_{\rm{L}}^{-1}$.
This ends the proof of theorem \ref{th:nonclassical}.

\section{\label{conclu}Conclusion}

We have shown that small amplitude oscillations in Newton's cradle are described by 
the DpS equation (\ref{dps}) over long times. From this result, we have estimated
on long time scales the maximal decay of small amplitude localized solutions
and proved the existence of stable long-lived breather states in Newton's cradle.

The justification of the DpS equation and the associated
estimates of maximal decay extend
straightforwardly to generalizations of (\ref{eq:cradle}) and (\ref{dps}) to arbitrary
space dimensions (i.e. for $n \in \mathbb{Z}^d$ and $d \geq 1$) when $x_n(t)$
defines a scalar field. However, generalizing our construction of long-lived breather states
would require an existence theorem for discrete breather solutions of the $d$-dimensional
DpS equation, which is not yet available for $d \geq 2$.
Other possible extensions of this work concern the generalization and justification
of the DpS equation when small spatial inhomogeneities are present in the original lattice (\ref{eq:cradle}),
as well as the addition of dissipative terms in (\ref{eq:cradle}) and (\ref{dps}).
Considering these effects is particularly important from a physical point of view 
when system (\ref{eq:cradle}) describes a granular chain \cite{JKC11}.

Other open problems concern the qualitative analysis of the DpS equation.
In particular, excitations generated from a localized disturbance and reminiscent of
traveling breather solutions have been numerically studied in \cite{JKC11,sahvm}, both
for the DpS equation and Newton's cradle. The existence of exact traveling breather solutions
of (\ref{dps}) is an open problem, and would imply 
(in the case of small amplitude waves) the existence of similar excitations in
Newton's cradle on long time scales. More generally, understanding in system (\ref{dps})
the complex mechanisms of fully nonlinear energy propagation from a localized disturbance
is a challenging open problem \cite{JKC11}. 
This would allow in particular to analyze the propagation of
nonlinear acoustic waves after an impact in granular chains with local potentials,
thanks to the connection we have established between (\ref{dps}) and
(\ref{eq:cradle1}).

\vspace{1ex}

\noindent
{\it Acknowledgements:}
This work is supported by the Rh\^one-Alpes Complex Systems Institute (IXXI).

\vspace{3ex}

\appendix
\begin{center}
{\Large\bf Appendix}
\end{center}

\section*{Simplified form of the amplitude equation}

This section provides
an explicit computation of the nonlinear term of the amplitude equation
(\ref{ampli}). Given the form (\ref{y0}) of $\mathcal{Y}_0$
and recalling that $P(X)=\zeta(X)e^{it}e_1+\cc$ (see lemma \ref{split}),
the amplitude $a(\tau)\in\ell_p (\mathbb{Z},\mathbb{C})$
satisfies the differential equation
$$
\partial_\tau a  =  \zeta(G_\alpha( \mathcal{Y}_0)),
$$
or more explicitly 
$$
i\partial_\tau a = \delta^+ f(\delta^- a),
$$
where 
\begin{eqnarray*}
\forall\, a\in \ell_p (\mathbb{Z},\mathbb{C}), &  & (f(a))_n = \tilde{f}(a_n ), \\
\forall\, a\in \mathbb{C}, & &
\tilde{f}(a)= 
\frac{1}{\sqrt2} \frac1{2\pi} \int_0^{2\pi} e^{-it} V'_\alpha \left(\frac{a\, e^{it}}{\sqrt2}+\cc\right) \rmd t.
\end{eqnarray*}
Below we compute the map $\tilde{f}$ explicitly.

\vspace{1ex}

Setting $a=r\, e^{i \theta}$ and using the change of variable $s=t+\theta$
in the integral defining $\tilde{f}$, one obtains
$$
\tilde{f}(a)=\frac{e^{i \theta }}{2^{3/2}\, \pi}
\int_{S^1}{ V_\alpha^\prime (\, \sqrt{2}\, r \cos{s}    \, )  \, e^{- i s}\, \rmd s},
$$ 
where one can fix $S^1=(-\pi ,\pi )$. Given the form (\ref{as:Hertz}) of the potential $V_\alpha$,
one obtains after elementary computations,
$$
\tilde{f}(a)=2^{\frac{\alpha -3}{2}}\, (k_- + k_+)\, c_\alpha \, a \, |a|^{\alpha -1},
$$
where
$$
c_\alpha = 
\frac{2}{\pi}
\int_{0}^{\pi /2}{ (\cos{t} )^{\alpha +1}   \, \rmd t}
$$
is a Wallis integral with fractional power $\alpha +1$.
Expressing $c_\alpha$ in terms of Euler's Gamma function leads to
$$
c_\alpha = 
\frac{1}{\pi}\frac{ \Gamma{(\frac{1}{2})}   \Gamma{(\frac{\alpha}{2}+1)}  }{  \Gamma{(\frac{\alpha +1}{2}+1)}   }
$$
(see \cite{abra}, formula 6.2.1 and 6.2.2, p. 258).
Since $\Gamma{(1/2)}=\sqrt{\pi}$ and $\Gamma{(a+1)}=a\, \Gamma{(a)}$, we obtain finally
$$
c_\alpha = 
\frac{
\alpha \, \Gamma{(\frac{\alpha}{2})}
}
{
\sqrt{\pi} (\alpha +1)  \Gamma{(\frac{\alpha +1}{2})}
}
.
$$


\end{document}